\documentclass[reqno,11pt]{amsart}
\pdfoutput1

\usepackage{amssymb,color,hyperref,mathrsfs,stmaryrd, tikz-cd}
\usepackage{amsmath}
\usepackage{mathtools}

\usepackage{xcolor}
\usepackage[curve,matrix,arrow]{xy}

\newcommand{\LL}{\mathcal{L}}
\newcommand{\F}{\mathcal{F}}
\newcommand{\C}{\mathcal{C}}
\newcommand{\M}{\mathcal{M}}
\newcommand{\N}{\mathcal{N}}
\renewcommand{\H}{\mathcal{H}}
\newcommand{\R}{\mathcal{R}}

\newcommand{\ov}{\overline}
\newcommand{\fto}{\longrightarrow}
\newcommand{\inv}{^{-1}}

\newtheorem{definition}{Definition}[section]
\newtheorem{theorem}{Theorem}[section]
\newtheorem{theoremlett}{Theorem}

\newtheorem{propositionlett}[theoremlett]{Proposition}
\newtheorem{lemma}[theorem]{Lemma}
\newtheorem{proposition}[theorem]{Proposition}

\theoremstyle{definition}
\newtheorem{counterex}[theorem]{Counterexample}
\newtheorem{example}[theorem]{Example}
\newtheorem{hypothesis}[theorem]{Hypothesis}
\newtheorem*{remark}{Remark}
\newtheorem*{notation}{Notation}
\newtheorem{referencing}[subsection]{}

\author{Edoardo Salati}
\title{Limits and colimits, generators and relations of partial groups}
\email{edoardo.salati@tu-dresden.de}

\begin{document}
\maketitle

\begin{abstract}
	We analyse limits and colimits in the category $Part$ of partial groups, algebraic structures introduced by A. Chermak. We will prove that $Part$ is both complete and cocomplete and, in addition, that the full subcategory of finite partial groups is both finitely complete and finitely cocomplete.\\
	Cocompleteness is then used in order to define quotients of partial groups. We will also identify a category richer than $Set$ (the category of sets and set-maps) and build the free partial groups over objects is such category; this yields a larger class of free partial groups, eventually allowing to prove that every partial group is the quotient of a free partial group.
\end{abstract}

\section*{Introduction}

The notion of a partial group was developed in connection with fusion systems; roughly speaking, a fusion system is a category whose objects are the subgroups of a finite $p$-group $S$ and whose morphisms are injective group homomorphisms, including at least all the conjugations by elements of $S$. Research on fusion systems is currently very active in connection to Local Finite Group Theory, Modular Representation Theory and Algebraic Topology. In this last field it has been a tool employed in the proof of the Martino-Priddy conjecture, asserting that, if $p$ is a prime, the $p$-completed classifying space associated to a finite group $G$ is homotopy equivalent to the classifying space associated to the fusion system $\F_S(G)$ (described, for example, in \cite[Part I, Section 1]{Aschbacher/Kessar/Oliver:2011}), where $S \in \mathrm{Syl}_p(G)$. With regard to this, partial groups played a prominent role in the definition of the classifying space associated to an abstract saturated fusion system $\F$; this was achieved by replacing $\F$ with a different, slightly richer, category, the so-called \emph{centric linking system}, and then defining the classifying space of $\F$ as the $p$-completed geometric realization of such associated centric linking system, whose existence and uniqueness are therefore necessary conditions to the definition. Partial groups provide a translation of the categorical language into one more similar to that of groups, which in turn allowed A. Chermak in \cite{Chermak:2013} to tackle and solve the problem by means of the well established group-theoretic machinery. The interested reader may find a beautiful survey in \cite{Aschbacher-Oliver:2016}, with Sections 3 and 4 covering the interplay between Algebra and Homotopy Theory.\\

There is a natural notion of morphisms of partial groups, as defined in \cite{Chermak:2022}, and in the same paper Chermak deals with further properties of the category $Part$ of partial groups, including completeness and cocompleteness. In particular, in \cite{Chermak:2022} the existence of colimits is proved exclusively over diagrams with specific properties; the main goal of this paper is to correct a mistake appeared in the proof of the cocompleteness of $Part$ in an unpublished early work of Chermak and, at the same time, prove the following result.
\begin{theoremlett}
	\label{theoremA}
	The category $Part$ is complete and cocomplete. In addition, if $FinPart$ is the full subcategory of $Part$ with objects the finite partial groups, then $FinPart$ is finitely complete and finitely cocomplete.
\end{theoremlett}
Early examples of how colimits in $Part$ may be exploited and used as a tool are the \emph{elementary expansions} defined by Chermak in \cite{Chermak:2016:II}, which consist in the main construction appearing in \cite[Section 5]{Chermak:2013} and building the centric linking system over a fusion system, as well as the construction realized by A. D\'{i}az Ramos, R. Molinier and A. Viruel in \cite{Ramos/Molinier/Viruel:2021} for proving that the category $Part$ is universal.\\
We will also take advantage of the categorical setting in order to define quotients for all partial groups and to study free partial groups: this will set the necessary background for the following result.
\begin{propositionlett}\label{prop:every-pg-is-quot-free}
	Every partial group is a quotient of a free partial group.
\end{propositionlett}
The notion of free partial groups over sets is analyzed very early, in Section \ref{section:intro}, as it is used to build a counterexample to Chermak's construction of colimits of partial groups. Such counterexample is presented in Section \ref{section:colimits}, together with a proof of the cocompleteness of $Part$. We will also report a proof of the completeness of $Part$ in Section \ref{section:limits}, although this can also be found in \cite{Chermak:2022}, with the purpose of collecting all coherent results in one single paper. Combining results of Sections \ref{section:colimits} and \ref{section:limits} yields a proof of Theorem \ref{theoremA}.\\
Existence of colimits is exploited in Section \ref{section:quotients} in order to define quotients of partial groups. Quotients were previously defined only for very specific subclasses, namely for groups and for the so-called localities, as defined in \cite[Sections 2 to 4]{Chermak:2022} and it will be proved that the definition given in this paper specializes to the already existing notions. Section \ref{section:gen and relations} is dedicated to the proof of Proposition \ref{prop:every-pg-is-quot-free}; as a consequence of the analysis undertaken in Section \ref{section:intro} it turns out that the notion of free partial groups over sets is too weak for such a result. After replacing the category $Set$ of sets and set-maps with a better suited category, denoted by $Set^s$, the notions of quotients and of free partial groups over $Set^s$ allow a proof in terms of generators and relations, in analogy with what happens for groups. The section is ended with an example of how these tools may be used: in detail, it is shown that Chermak's elementary expansions construction may be realized in terms of generators and relations.

\section*{Acknowledgements}
First of all, to my supervisor Ellen Henke, who has introduced me to the world of fusion systems, partial groups and localities, many thanks for reading the paper and for the several pieces of advice she provided me with.\\
A special thank to Andrew Chermak, for reading the results and advancing the idea of publishing them in a paper. He was the first to deal with the categorical properties of partial groups and this paper, actually, stems from previous work of his.\\
Thanks also to A. Gonz\'{a}lez and to A. D\'{i}az Ramos, R. Molinier and A. Viruel for their interest, the emails we exchanged and the suggestions I received from them.

\section{Partial groups and free partial groups}
\label{section:intro}
		
Roughly speaking, a \emph{partial group} is a set $\LL$ equipped with a \emph{partial} operation, defined on a subset of the set of words on the alphabet $\LL$; in other words, it may not be possible to multiply all elements as it happens with groups. \\
A little remark about the notation adopted throughout the paper: we will use right-hand notation for functions and functors, i.e. writing $(x)f$, or simply $xf$, whenever $f$ is a function or a functor and $xf$ is the object (or morphism) associated to $x$ by $f$. Composition of functions or functors will then be denoted by simple juxtaposition or by the symbol $\cdot$.\\

\begin{definition}
	\label{def partial gr}
	Let $\LL$ be a non-empty set and $W(\LL)=W$ be the free monoid on $\LL$; denote concatenation of words in $W$ by the symbol $\circ$. Consider a subset $D=D(\LL) \subseteq W$, a map $\Pi: D \rightarrow \LL$, which will be referred to as the \emph{multivariable product} of $\LL$ (and sometimes just as the \emph{product} of $\LL$), and an involutory bijection $i: \LL \rightarrow \LL$.
	Then the quadruple $(\LL, D, \Pi, i)$ (often denoted just by $\LL$ when the context leaves no ambiguity) is a partial group provided that the following hold:
	\begin{enumerate}
		\item $\LL \subseteq D \qquad$ and moreover $ \qquad u \circ v \in D \quad \implies \quad u,v \in D$;
		\item $\Pi |_{\LL} = id_{\LL}$;
		\item if $u \circ v \circ w \in D$, then $\quad u \circ (v)\Pi \circ w \in D \quad$ and $\quad (u \circ v \circ w)\Pi = (u \circ (v)\Pi \circ w)\Pi$;
		\item by extending $i$ to $W(\LL)$ defining $(x_1, \dots, x_n)i = ((x_n)i, \dots, (x_1)i)$, if $w \in D$ then we have $ \quad (w)i \circ w \in D \quad$ and $\quad ((w)i \circ w) \Pi = 1_\LL = (\emptyset)\Pi$, where $\emptyset$ is the empty word.
	\end{enumerate}
	In analogy with terminology for groups, we call the element $1_\LL$ the \emph{unit} or \emph{neutral element} of the partial group $\LL$ and for every $x \in \LL$, $(x)i$ is the \emph{inverse} of $x$.\\
	The partial group $(\LL, D, \Pi, i)$ is \emph{finite} if the set $\LL$ is finite.
\end{definition}

For a survey of some elementary properties of partial groups we refer to \cite[Section 1]{Chermak:2022}. Clearly, $\LL$ is a group if and only if $D(\LL)=W(\LL)$ and, in such a case, $\Pi$ gives the group multiplication thanks to associativity. \\
A. Chermak also defines a notion of morphisms of partial groups.

\begin{definition}
	\label{def morphisms part gr}
	Given partial groups $(\LL, D, \Pi, i)$ and $(\LL', D', \Pi', i')$, a morphism $\beta$ of partial groups from $\LL$ to $\LL'$ is a set-wise map $\beta: \LL \rightarrow \LL'$ such that, if $\beta^* : W(\LL) \rightarrow W(\LL')$ is the componentwise extension, i.e. $(x_1, \dots, x_n)\beta^*:=((x_1)\beta, \dots, (x_n)\beta)$, then:
	\begin{itemize}
		\item[(a)] $(D)\beta^* \subseteq D'$;
		\item[(b)] we have a commutative diagram
		\begin{center}
		\begin{tikzcd}
			D \ar[r,"\Pi"] \ar[d, "\beta^*"] & \LL \ar[d, "\beta"] \\
			D' \ar[r, "\Pi '"]	& \LL'.
		\end{tikzcd}
		\end{center}
	\end{itemize}
\end{definition}

We therefore obtain a category $Part$ whose objects are partial groups and whose morphisms are morphisms of partial groups according to the above definition, clearly with the usual composition of maps. It is trivial to check that $\beta$ is an isomorphism (in the categorical sense) if and only if $\beta$ is a bijective morphism of partial groups satisfying $(D)\beta^* = D'$. Moreover, morphisms of partial groups behave well with respect to units, that is if $\beta: (\LL,D,\Pi,i) \fto (\LL',D',\Pi',i')$ is a morphism between partial groups, then $(1_\LL)\beta = 1_{\LL'}$, as shown in \cite[Lemma 1.13]{Chermak:2022}. More generally, there is a fully-faithful inclusion functor $Grp \hookrightarrow Part$.\\

\begin{notation}
For any map $f: X \fto Y$ between sets $X,Y$, the notation $f^*$ will always denote the induced map $f^*:W(X) \fto W(Y)$ (or restrictions of it) as defined in Definition \ref{def morphisms part gr}. This applies also for maps between partial groups or other structures with a forgetful functor in $Set$, up to applying such forgetful functor. On the other hand, given a partial group $(\LL,D,\Pi, i)$, in accordance to the notation in Definition \ref{def partial gr} the symbol $i$ will denote also the extension of $i$ to $W(\LL)$ taken by reversing order. We will also equivalently write $D(\LL)$ or $D_\LL$ for the domain of a partial group $\LL=(\LL,D_\LL,\Pi,i)$.
\end{notation}

Define the category of pointed sets, $Set^*$, as the category whose objects are pairs $(X,x)$, where $X$ is a non-empty set and $x \in X$, and whose morphisms $f: (X,x) \fto (Y,y)$ are set-maps $f:X \fto Y$ with the property that $(x)f=y$. Composition is given by the usual composition of maps. Since partial groups are naturally pointed at the unit, we will work with pointed sets instead of sets; this will simplify the construction of free partial groups. In better words, let's consider first the forgetful functor $U : Grp \fto Set$; we can actually define a slightly different forgetful functor (still denoted by $U$) which remembers a little more, namely
\[
U: Grp \fto Set^*, \quad (G)U = (G,1_G)
\]
where $1_G$ is the unit of the group. This is possible since $(1_G)f=1_H$ whenever $f : G \fto H$ is a homomorphism between groups. Since morphisms of partial group have the same property, we can similarly define a forgetful functor
\[
U: Part \fto Set^*, \quad (\LL,D,\Pi,i)U = (\LL, 1_{\LL}).
\]
It is a classical situation that in which $U$ is a right adjoint to the free-construction functor; indeed, we will prove that $U: Part \fto Set^*$ is a right adjoint by building free partial groups over pointed sets. It will become clear that the construction may be carried in a very similar way also when dealing with the category $Set$ in place of $Set^*$ and yields the same partial group.

\subsection{Free partial groups over pointed sets}
Suppose we have a pointed set $(X,1)$ (the base-point will become the unit of the partial group). Consider the set $Y := \{1\} \sqcup X^* \sqcup \hat{X}^*$, where $X^* = X \setminus \{1\}$, $\sqcup$ denotes the disjoint union and we use $\hat{X}^*$ to distinguish the two copies of $X^*$ (the same notation will also be adopted for the respective elements).
Define an involutory bijection $i$ on $Y$ by $(1)i = 1$ and $(x)i= \hat{x}$ and set $D \subseteq W(\LL)$ to be the set of words obtained from alternating finite strings of the form $(\dots, x, \hat{x}, x, \hat{x}, \dots)$, for $x \in X^*$, by adding in any position any finite number of copies of the element $1$. We say that such strings are built on the element $x$.\\
For $w \in D$ built on the element $x \in X^*$, define
\[
(w)\Pi =
\begin{cases}
	x \quad \text{ if the number of $x$ is greater than that of $\hat{x}$}, \\
	1 \quad \text{ if the number of $x$ equals that of $\hat{x}$}, \\
	\hat{x} \quad \text{ if the number of $x$ is lower than that of $\hat{x}$}.
\end{cases}
\]
We now prove that $(Y, D, \Pi, i)$ is a partial group. Clearly $Y \subseteq D$ and $\Pi|_Y = \mathrm{id}_Y$; if $u,v$ are words in $W(Y)$ such that $u \circ v \in D$, then it is a string built on some element $x \in X^*$ and, therefore, also $u$ and $v$ are strings built on the same element $x$, yielding $u,v \in D$. These are axioms (1) and (2) of Definition \ref{def partial gr}. To prove axiom (3), first of all note that we may identify the $\hat{-}$ notation with application of $i$. Suppose that $u \circ v \circ w \in D$; if $v$ is the empty word or a word made only of the element $1$, there is nothing to prove, otherwise, supposing that $v$ has the form $v=(a_0,\dots, b_0)$ with $a_0,b_0 \in \{1,x,\hat{x}\}$, we may find elements $a, b \in \{x, \hat{x}\}$ so that $v=(1,\dots, 1, a, \dots, b, 1, \dots, 1)$. Then the last occurrence of an element which is not $1$ in $u$ must be $\hat{a}$ and the first occurrence of an element which is not $1$ in $w$ must be $\hat{b}$. If $a=x$, then $(v)\Pi \in \{1,x\}$; when $(v)\Pi = 1$, $b=\hat{x}$, so $\hat{a}=\hat{x}$ and $\hat{b}=x$; this implies $u \circ (v)\Pi \circ w \in D$. When $(v)\Pi = x$, then $b=x$ and in this case we have $\hat{a}=\hat{x}$, $(v)\Pi=x$ and $\hat{b}=\hat{x}$, showing again that $u \circ (v)\Pi \circ w \in D$. The case $a=\hat{x}$ is analogous; that $(u \circ v \circ w)\Pi=(u \circ (v)\Pi \circ w)\Pi$ is trivial, so axiom (3) holds. Let now $u = (u_0, \dots, u_n) \in D$ and suppose it is built on an element $x \in X^*$; then $(u)i \circ u=((u_n)i, \dots, (u_0)i, u_0, \dots, u_n)$ and, by definition of $i$, also $(u)i$ is built on $x$. Let $u_k$ be the first occurrence of an element different from $1$ in $u$, then $(u_k)i$ is the last occurrence of element different from $1$ in $(u)i$, showing that $(u)i \circ u$ is built on $x$. Thus $(u)i \circ u \in D$. Finally, let $d$ be the difference between the number of occurrences of $x$ and the number of occurrences of $\hat{x}$ in $u$. Then $d \in \{-1,0,1\}$ and the same difference computed on $(u)i$ gives $-d$; this shows that $((u)i\circ u)\Pi = 1$, so axiom (4) also holds.\\

Define an association $F$ by setting $(X,1)F := (Y,D,\Pi,i)$. In order to obtain a functor it remains to define $F$ over morphisms; let $f: (X,1_X) \rightarrow (Z,1_Z)$ be a morphism of pointed sets, then we associate to $f$ the map $fF: (X,1_X)F \rightarrow (Z,1_Z)F$ defined by setting $(1_X)fF=1_Y$, $(x)fF=(x)f$ and $(\hat{x})fF=\widehat{(x)f}$. It is straightforward to check that $fF$ is a partial group morphism and that $F$ is indeed a functor. A similar construction can be carried over sets and in this case one ends with the free partial group over the set $X$ being exactly $(X \sqcup \{1\}, 1)F$.\\
Let $U: Part \fto Set^*$ be the forgetful functor defined by $(\LL,D,\Pi,i)U=(\LL,1_\LL)$; for every pointed set $(X,1)$ there exists a natural inclusion $j: (X,1) \hookrightarrow (X,1)FU$ simply defined by $(x)j=x$ for every $x \in X$. Then $(X,1)F$ is the free partial group over the pair $((X,1),j)$; indeed for any pointed-set map $f: (X,1) \rightarrow (\LL,1_\LL)$, where $\LL$ is a partial group, there exists a unique morphism of partial groups $\tilde{f} : (X,1)F \fto \LL$ making the following diagram commute.
\begin{center}
\begin{tikzcd}
	(X,1) \ar[r, "j"] \ar[d, "f"] & (X,1)FU \ar[dl, dotted, "(\tilde{f})U"]\\
	(\LL)U 
\end{tikzcd}
\end{center}
$\tilde{f}$ is defined by $(1)\tilde{f} = 1_\LL$, $(x)\tilde{f}=(x)f$ and $(\hat{x})\tilde{f} \mapsto ((x)f)^{-1}$, where $x \in X$ and the exponent $-1$ denotes the inverse in $\LL$. Note that the existence of the morphism $\tilde{f}$ is guaranteed by the fact that the domain $(X,1)F$ is the \emph{smallest} possible satisfying axioms (1) to (4) of Definition \ref{def partial gr} (this will become formally clear in Lemma \ref{free-part-gr-str-set}), therefore universal initial with respect to morphisms of partial groups. As an immediate consequence we have bijections
\begin{equation*}
	\label{adjunction}
	\tag{$\prime$}
	Hom_{Part}((X,1)F, \LL) \cong Hom_{Set^*}((X,1), (\LL)U)
\end{equation*}
yielding the following lemma.

\begin{lemma} \label{adjoint functors free-forgetful-first}
	With the notation above, $F \dashv U: Part \fto Set^*$.
	\begin{proof}
		We explicitly define a natural isomorphism between the bifunctors $Hom_{Part}((-)F,-)$ and $Hom_{Set^*}(-, (-)U)$. Let $(X,1) \in Ob(Set^*)$ and $\LL=(\LL,D,\Pi,i) \in Ob(Part)$ and define
		\begin{center}
		\begin{tikzcd}
		\psi_{((X,1),\LL)}: &[-25pt] Hom_{Set^*}((X,1), (\LL)U) \ar[r] & Hom_{Part}((X,1)F,\LL) \\ [-25pt]
			&			f \ar[r, mapsto] & \tilde{f}
		\end{tikzcd}
		\end{center}
		where $\tilde{f}$ is the map defined above; then (\ref{adjunction}) shows that $\psi_{((X,1),\LL)}$ is a bijection.\\
		We need to show naturality; let $g: (Y,1) \fto (X,1)$ be a morphism of pointed sets, $h: (\LL,D(\LL), \Pi_\LL, i_\LL) \fto (\M, D(\M), \Pi_\M, i_\M)$ be a morphism of partial groups and consider the following diagram 
		\begin{center} 
		\begin{tikzcd}
		Hom_{Set^*}((X,1),(\LL)U) \ar[d] \ar[r,"\psi_{((X,1),\LL)}"] &[+20pt]	Hom_{Part}((X,1)F,\LL) \ar[d] \\
		Hom_{Set^*}((Y,1),(\M)U) \ar[r,"\psi_{((Y,1),\M)}"] &	Hom_{Part}((Y,1)F,\M)
		\end{tikzcd}
		\end{center}
		where the vertical arrows are pre- and post-composition with the morphisms $g$ and $h$ on the left side, $gF$ and $hF$ on the right side.
		Let $f \in Hom_{Set^*}((X,1),(\LL)U)$; following the upper-right path we get
		\[
		f \mapsto \tilde{f} \mapsto gF \cdot \tilde{f} \cdot h,
		\]
		whereas the bottom-left path yields
		\[
		f \mapsto g \cdot f \cdot hU \mapsto \widetilde{g \cdot f \cdot hU}.
		\]
		By (\ref{adjunction}) these morphisms are determined by their image on $1$ and $Y$, and for $y \in Y \cup \{1\}$ we get
		\[
		(y)gF \cdot \tilde{f} \cdot h = (y)g \tilde{f} \cdot h = ((y)gf)h,
		\]
		and
		\[
		(y)\widetilde{g \cdot f \cdot hU} = (y)(g \cdot f \cdot hU) = ((y)gf)h.
		\]
		This completes the proof.
	\end{proof}
\end{lemma}

Note that neither $UF$ nor $FU$ are the identity functor; as a consequence even though $F$ preserves colimits (so we can push colimits in $Set^*$ to colimits in $Part$), we cannot use this fact to prove that $Part$ is cocomplete as there is a change of objects.

\section{Colimits in $Part$}
\label{section:colimits}

\subsection{Colimits in $Set^*$}\label{colim in Set-*}
Before discussing about colimits in $Part$, we shall study in detail colimits in the category $Set^*$. Recall that given a relation $R$ on a set $X$ there exists a unique smallest equivalence relation on $X$ containing $R$, exactly the intersection of all equivalence relations on $X$ containing $R$. In addition, the functions with domain $X$ that factor through the relation $R$ are exactly the functions factoring through the smallest equivalence relation containing $R$. In detail,
\begin{lemma}
	\label{propos:factor on relation}
	Let $f:X \rightarrow Y$ be a function between sets and $\sim'$ a binary relation on $X$, i.e. $\sim' \subseteq X \times X$. Let $\sim$ be the smallest equivalence relation on $X$ containing $\sim'$. Then $f$ factors through $X/ \sim$, as in the commutative diagram below, if and only if $(u)f = (v)f$ for every $u \sim' v$.
	\begin{center}
	\begin{tikzcd}
		f: &[-15pt] X \ar[r,"f"] \ar[d,"\pi"] & Y \\
					& \frac{X}{\sim} \ar[ru,dashed,"\exists ! \hat{f}"']
	\end{tikzcd}
	\end{center}
	
\begin{proof}
	Since $\sim' \subseteq \sim$, the \emph{only if} part is trivial. As for the \emph{if} part,
	sharing the same image by $f$ is an equivalence relation, that is we may define
	\[
	\forall u,v \in X, \; u \sim'' v \iff (u)f = (v)f.
	\]
	Then we have $\sim' \subseteq \sim''$ and, as $\sim$ is the smallest equivalence relation containing $\sim'$, also $\sim \subseteq \sim''$; clearly $f$ factors through $\sim''$, thus it also factors through $\sim$.
\end{proof}
\end{lemma}

Turning our attention to colimits, let $\C$ be a small category and $F:\C \rightarrow Set^*$ a functor; we call such a functor a \emph{diagram} of type $\C$ in $Set^*$. For every $X \in Ob(\C)$ denote by $(X,x)$ its image under $F$; similarly for every $j \in Mor_{\C}(X,Y)$ denote by $jF$ the image of the morphism $j$ via $F$. In order to avoid an excessively elaborate notation, we will write disjoint unions of pointed sets, meaning by it the disjoint union of the sets and forgetting about the base-points, the outcome being a set (not pointed). Then one has
\begin{equation*}
\label{colimit}
\tag{$\star$}
colim(F) = \left( \dfrac{\displaystyle \bigsqcup_{X \in Ob(\C)} (X)F}{\sim}, [x] \right)
\end{equation*}
where $\sim$ is the smallest equivalence relation on $\bigsqcup_{X \in Ob(\C)} (X)F$ containing all the pairs $(x,y)$ and $(z,(z)jF)$ for all images $XF=(X,x), \, YF=(Y,y)$ and all $z \in (X,x)$ and $j \in Mor_{\C}(X,Y)$; here $[x]$ denotes the equivalence class of a (and therefore any) base-point $x$. The maps building a cocone to $colim(F)$ are defined as $i_X: (X,x) \fto colim(F), \, z \mapsto [z]$. Indeed consider a cocone $((R,r), t_X)$ from $F$ to $(R,r)$
\begin{center}
\begin{tikzcd}
	(X,x) \ar[rd,"i_X"] \ar[rrd,bend left=25,"t_X"] \\
		& colim(F) \ar[r,dashed,"\exists ! \psi"] & (R,r) \\
	(Y,y) \ar[ru,"i_Y"] \ar[rru,bend right=25,swap,"t_Y"]
\end{tikzcd}
\end{center}
and define the morphism $\psi$ by $[z] \overset{\psi}{\rightarrow} (z)t_X$ for $z \in (X,x)$.
In order to prove that $\psi$ is well-defined we need to show that it factors through $\sim$; in virtue of Lemma \ref{propos:factor on relation} we only need to prove $(z)t_X = ((z)jF)t_Y$ for any $z \in (X,x)$ and $j \in Mor_\C(X,Y)$, but this is trivially true by the definition of a cocone. As for the uniqueness, any other morphism $\varphi$ such that $i_X \cdot \varphi=t_X$ for every $X \in Ob(\C)$ satisfies $([z])\varphi= (z)i_X \cdot \varphi = (z)t_X$, thus $\varphi = \psi$.

\subsection{Factorization of the multivariable product}
\label{factoriz of multivar product}
We describe a possible attempt to prove cocompleteness of $Part$ through colimits of $Set^*$. If $\C$ is a small category, $a \in Ob(\C)$ and $T: \C \rightarrow Part$ a functor, denote images via $T$ by $aT=(aT, D(aT), \Pi_{a}, i_{a})$ and, for each morphism $x: a \rightarrow b$ in $\C$, its image under $T$ by $\lambda _x: aT \rightarrow bT$. Now consider the forgetful functor $U: Part \rightarrow Set^*$ and set $(\LL,1_\LL):=colim(TU)$ to be the colimit in the category of pointed sets. We wish to define a partial group structure over $\LL$ which makes it the colimit over $T$ in $Part$.\\
Since we will often refer to the following way of defining a domain, we provide a numbering for referencing reasons.
\begin{referencing}
\label{definition:of domain}
A natural way to define the domain $D \subseteq W(\LL)$ consists in taking $D$ as the set of words $[w]=([w_1], \dots, [w_n])$ for which there exist some $a \in Ob(\C)$ such that $w_i \in aT$ for all $i \in \{1, \dots, n\}$ and $(w_1, \dots, w_n) \in D(aT)$, where square brackets $[-]$ are used to denote the equivalence classes of the relation defining $\LL$. Equivalently, if $\rho$ is the projection over the quotient defining $\LL$, as in (\ref{colimit}), we may restrict $\rho^*$ to $\bigsqcup_{a \in Ob(\C)} D(aT)$, that is $(u_1, \dots,u_n)\rho^* = ((u_1)\rho, \dots,(u_n)\rho)$ for any word $(u_1,\dots,u_n) \in D(aT)$ and any $a \in Ob(\C)$: then $\displaystyle D=\left( \bigsqcup_{a \in Ob(\C)} D(aT) \right) \rho^*$.
\end{referencing}
Clearly $D$ satisfies condition (1) of Definition 1. One would now hopefully define a product $\Pi$ by lifting to products $\Pi_a$ of the partial groups $aT$ for suitable objects $a \in Ob(\C)$; for this purpose consider the following diagram
\begin{equation*}
\label{diagram-fact-multprod}
\tag{$\dagger$}
	\begin{tikzcd}
	\displaystyle \bigsqcup_{a \in Ob(\C)} D(aT) \ar[d,"\rho^*"] \ar[r,"\sqcup \Pi_a"] & \displaystyle \bigsqcup_{a \in Ob(\C)} aTU \ar[r,"\rho"] & (\LL,1_\LL) \\
	D \ar[rru, dotted, bend right=15, "\Pi"]
	\end{tikzcd}
\end{equation*}
One might now hope to factor $\rho^*$ via $\sqcup \Pi_a \cdot \rho$.
This argument would at least be complicated to prove and we will actually show it wrong. Indeed suppose we have $u=(u_1,\dots,u_k) \in D(aT)$ and $v=(v_1,\dots,v_h) \in D(bT)$ such that $u \rho^* = v \rho^*$, then we get that $u$ and $v$ have equal length and for all the components we have $[u_i] = [v_i]$. Now we need $[(u)\Pi_a] = [(v)\Pi_b]$, however since there is no guarantee that the components $u_i$ and $v_i$ are identified through a unique morphism of partial groups $aT \rightarrow bT$, we cannot use property (b) of Definition \ref{def morphisms part gr} to prove our factorization. \\
We will provide a counterexample showing how a construction carried as above may, in general, fail; in order to build a not too complicated counterexample we will need to split the formation of colimits according to the following result about categories.

\begin{proposition}
	\label{prop:co complete cat}
	Consider a locally small category $\C$; then every (co)limit is the (co)equalizer of a (co)product. In particular $\C$ is (finitely) (co)complete if and only if $\C$ admits all (finite) (co)products and all (co)equalizers.
	\begin{proof}
		See \cite[Chapter V, section 2]{MacLane:1971a}, where the author proves the statement for limits; the statement for colimits is obtained by duality.
	\end{proof}
\end{proposition}

We start by showing that $Part$ has coproducts, since the factorization discussed above works in this case and fails when dealing with coequalizers.

\begin{lemma}
	\label{part has coproducts}
	The category $Part$ has all coproducts.
\begin{proof}
	Let $\C$ be a small, discrete category, $T$ a diagram of type $\C$ in $Part$, $U$ the forgetful functor $Part \fto Set^*$. Our aim is to show that there exists a partial group structure over the coproduct taken in $Set^*$, that is $colim(TU)$, which makes it the coproduct in $Part$. We will extensively refer to the construction described in Section \ref{factoriz of multivar product} and adopt the same notation. For every $a \in Ob(\C)$ set $\LL_a = (\LL_a,D(\LL_a),\Pi_a,i_a) := aT$ and write $1_a$ for its unit. From Section \ref{colim in Set-*} we obtain
	\[
	colim(TU) = \displaystyle \frac{\bigsqcup_{a \in Ob(\C)} (\LL_a)U}{\sim},
	\]
	where $\sim$ is the equivalence relation identifying exactly all the units $1_a$. Set also $\LL := colim(TU)$ and define the domain $D(\LL)$ of $\LL$, as well as the maps $\rho$ and $\rho^*$, as in \ref{definition:of domain}.	Let $u \in D(\LL_a)$ and $v \in D(\LL_b)$ be such that $(u)\rho^* = (v)\rho^*$, then either all the components of $u$ and $v$ are the unit element or $\LL_a = \LL_b$ and $u = v$. Hence, in this situation $\bigsqcup \Pi_a \cdot \rho$ factors through $\rho^*$; we take the resulting map $\Pi: D(\LL) \rightarrow \LL$ as the product on $\LL$.
	The inversion map $i$ on $\LL$ is defined by $([z])i = [(z)i_a]$ when $z \in \LL_a$; note that for $[z]=[1_a]$, $([1_a])i = [(1_a)i_a] = [1_a] = [1_b] = ([1_b])i$ for any $a,b \in Ob(\C)$, so $i$ is well-define. \\
	We have already seen in Section \ref{factoriz of multivar product} that axiom (1) of Definition \ref{def partial gr} holds; axiom (2) is, also, trivial. Suppose now that $u \circ v \circ w \in D(\LL)$, then we can lift via $\rho^*$ to a word in $D(\LL_a)$ for some $a \in \C$ and remap via $\sqcup \Pi_a \cdot \rho$ to $\LL$; by the commutativity of (\ref{diagram-fact-multprod}) we obtain $(u \circ v \circ w)\Pi$. Since $\LL_a$ is a partial group, axiom (3) holds for the lifting of $u \circ v \circ w$ and, therefore, holds also in $\LL$. Similarly also axiom (4) holds.\\
	We have natural inclusion morphisms $j_a: \LL_a \rightarrow \LL$, given by $(x)j_a = [x]$, which are morphisms of partial groups and build a cocone $(\LL, j_a)$ from $T$. Then $(\LL, j_a)$ is the coproduct; indeed consider another cocone $(\M, f_a)$ from $T$, with $\M=(\M,D(\M), \Pi_\M, i_\M)$,
	\begin{center}
	\begin{tikzcd}
		\LL_a \ar[rd, "j_a"] \ar[rrd,bend left=20, "f_a"] \\
			&	\LL \ar[r, dashed, "\exists ! \psi"] & \mathcal{M} \\
		\LL_b \ar[ru, "j_b"] \ar[rru,bend right=20, "f_b"]
	\end{tikzcd}
	\end{center}
Then all we need to show is that the pointed-set morphism $\psi$ defined in Section \ref{colim in Set-*}, acting as $([x])\psi = (x)f_a$ for $x \in \LL_a$, is a morphism of partial groups. Note that $D(\LL_a)f_a \subseteq D(\mathcal{M})$, so that also $D(\LL) \subseteq D(\mathcal{M})$; if $u=([u_1], \dots, [u_k]) \in D(\LL)$, then every class $[u_k]$, with $u_k$ not a unit of some $\LL_a$, is a singleton, whereas the units form all together a single class. Thus all the $u_k$ appearing in $u$ which are not a unit must all be elements of the same $\LL_a$ for some $a \in Ob(\C)$. By picking the units to be $1_a$, we can lift $u$ to a word $(u_1, \dots, u_n) \in D(\LL_a)$. Consider now the following diagram
\begin{center}
\begin{tikzcd}
D(\LL) \ar[d,"\psi^*"] \ar[r,"\Pi"] &	\LL \ar[d,"\psi"] \\
D(\M) \ar[r,"\Pi_\M"] & 	\M \\
D(\LL_a) \ar[u,"f_a^*"] \ar[r,"\Pi_a"] &	\LL_a \ar[u,"f_a"]
\end{tikzcd}
\end{center}
where the bottom square is commutative since $f_a$ is a morphism of partial groups.
Now $(u)\Pi \cdot \psi = (u_1,\dots,u_n) \Pi_a \cdot \rho \cdot \psi = ([(u_1,\dots,u_n)\Pi_a])\psi$. On the other hand
\[
(u)\psi^* \cdot \Pi_\M = (([u_1])\psi,\dots,([u_n])\psi)\Pi_\M = ((u_1)f_a, \dots, (u_n)f_a) \Pi_\M = (u_1,\dots,u_n)f_a^* \cdot \Pi_\M.
\]
Since $f_a=j_a \cdot \psi$, we have 
\[
(u)\psi^* \cdot \Pi_\M  = (u_1,\dots,u_n)\Pi_a \cdot f_a = (u_1,\dots,u_n)\Pi_a \cdot j_a \cdot \psi = (u_1,\dots,u_n)j_a^* \cdot \Pi \cdot \psi = ([(u_1,\dots,u_n)\Pi_a])\psi,
\]
so also the upper square is commutative. This completes the proof.
\end{proof}
\end{lemma}

\begin{remark}
	Note that coproducts of a finite number of finite partial groups are again finite; this will translate, after proving existence of coequalizers, in the fact that the full subcategory of $Part$ of finite partial groups is finitely cocomplete. This is probably among the major differences between the categories $Part$ and $Grp$.
\end{remark}

\subsection{Coequalizers of partial groups} \label{coequaliz partial grous}
We now begin our analysis of coequalizers; in order to keep notation simple, we omit writing the application of the forgetful functor $U: Part \fto Set$ whenever it is clear by the context whether we are considering a partial group or just its underlying set.\\
Consider two partial groups $\LL=(\LL,D(\LL), \Pi_\LL, j)$ and $\LL'=(\LL',D(\LL'),\Pi, i)$ and two parallel morphisms of partial groups $f,g: \LL \rightarrow \LL'$. The setwise coequalizer is given by
\[
C:=coeq_{Set}(f,g) = \frac{\LL'}{\sim}
\]
paired with the quotient map $q: \LL' \rightarrow (\LL' / \sim) $, where $\sim$ is the smallest equivalence relation on $\LL'$ containing the pairs $(xf,xg)$ for $x \in \LL$. This is a simplified realization of the same object constructed in $\eqref{colimit}$; according to this we are supposed to consider $(\LL \sqcup \LL')/\mathcal{R}$, with $\mathcal{R}$ the smallest equivalence relation containing all pairs $(x,xf)$ and $(x,xg)$ for $x \in \LL$. In particular $\mathcal{R}$ identifies all $xf$ with $xg$; it is a simple exercise to prove that there is a bijection
\[
\frac{\LL'}{\sim} \cong \frac{\LL \sqcup \LL'}{\mathcal{R}}
\]
which is the identity over representatives in $\LL'$. \\
We try to define a partial group structure on $C$; for the purpose, let $q: \LL' \longrightarrow C$ be the canonical projection and $D=D(C)$ be defined as in Section \ref{definition:of domain}, so in this case as the set of words $\ov{w}$ which are the pointwise projection via $q$ of a word $w \in D(\LL')$. A word of caution: the element $\ov{w}$ is not an equivalence class with representative $w$, but a word whose single entries are elements of $C$, so equivalence classes with representatives in $\LL'$. We can picture the situation through the following diagram
\begin{center}
\begin{tikzcd}
	D(\LL') \ar[d,"\Pi"] \ar[r,"q^*"]	& D(C) \ar[d,dashed, "\pi"] \\
	\LL' \ar[r,"q"]		& C
\end{tikzcd}
\end{center}
Commutativity of this diagram is a necessary condition since $q$ has to be a morphism of partial groups. Thus the multivariable product $\pi: D(C) \rightarrow C$ of $C$ has to be defined as $[w]\pi=(w)\Pi q$; similarly, the inversion $l$ on $C$ must be given by $[w]l=(w)i q^*$.\\
That $l$ is well-defined is an immediate consequence of the fact that $f$ and $g$ are partial group morphisms, so they send inverses to inverses; the situation with $\pi$ is instead much more complex. The possibility of factorizing $\Pi q$ via the map $q^*$ is equivalent to proving the implication $(u)q^* = (v)q^* \implies u\Pi \sim v\Pi$, where $u,v \in D(\LL')$. The relation $\sim$ is generated by the weaker relation $\sim'$ made exactly of the pairs $(xf,xg)$ for $x \in \LL$; being $fq=gq$ we get
\[
(xf)\Pi q = (x\Pi_\LL)fq  = (x\Pi_\LL)g q = (xg)\Pi q,
\]
that is $(xf)\Pi \sim (xg)\Pi$. Combining with Proposition \ref{propos:factor on relation} we obtain that $\pi$ is well defined on words of length 1.\\
Problems arise when considering words of length greater than 1; if we can write $u=xf^*=(x_1f, \dots, x_nf)$ and $v=xg^*=(x_1g, \dots, x_ng)$ for $x=(x_1,\dots,x_n) \in D(\LL)$, then we clearly get
\[
u \Pi = (x)f^* \Pi= x \Pi_\LL f \sim x \Pi_\LL g = (x)g^* \Pi = v \Pi. 
\]
However if $u=(u_1, \dots, u_n)$ and $v=(v_1, \dots, v_n)$ are words in $D(\LL')$ such that $uq^*=vq^*$, we only know that $u_i \sim v_i$ for each $1 \le i \le n$. So even if we are able to represent $uq^*$ through an element of the form $\tilde{u}=((x_i) h_i)$, where $h_i=f$ for some $i$ and $h_j=g$ for $j \ne i$, there is no guarantee that $u\Pi \sim \tilde{u} \Pi$. To better understand the obstruction to our factorization suppose $n=2$, so $u=(u_1,u_2)$ and $v=(v_1,v_2)$, with $u_1=xf$, $v_1=xg$, $u_2=yf$ and $v_2=yg$. Clearly the pairs $(xf,yf), (xf,yg), (xg,yg)$ and $(xg,yf)$ are all representatives of the same word in $D(C)$; we have already observed that $(xf,yf)\Pi \sim (xg,yg)\Pi$, however there is no evident reason, for example, for $(xf,yf)\Pi \sim (xf,yg)\Pi$. The counterexample that follows is constructed by following precisely the same line of reasoning.\\

\begin{counterex}
\label{counterex to factoriz}
Consider the pointed set $(A,1)$ with $A= \{ 1,a,b \}$ and let $\LL$ be the free partial group on $(A,1)$; recall that $\LL = \{1, a, \hat{a}, b, \hat{b} \}$. Consider the Klein group $\displaystyle M= \frac{\mathbb{Z}}{2 \mathbb{Z}} \times \frac{\mathbb{Z}}{2 \mathbb{Z}}$ ($(M,1)$ as a pointed-set) with generators $(x,1)=x$ and $(1,y)=y$ and define the pointed-set functions
\begin{center}
	\begin{tikzcd}
	\tilde{f} : &[-25pt] (A,1) \ar[r] & M \\ [-15pt]
	& a \ar[r,mapsto]  & x  \\ [-20pt]
	& b \ar[r, mapsto] & y
	\end{tikzcd}
\end{center}
\begin{center}
	\begin{tikzcd}
	\tilde{g} : &[-25pt] (A,1) \ar[r] & M \\ [-15pt]
	& a \ar[r,mapsto]  & xy  \\ [-20pt]
	& b \ar[r, mapsto] & x
	\end{tikzcd}
\end{center}
Then we get induced morphisms of partial groups $f,g: \LL \fto M$ making the diagram below commutative
\begin{center}
	\begin{tikzcd}
	(A,1) \ar[rr,hook] \ar[dr,shift left=.80,"\tilde{g}"] \ar[dr,shift right=.80,"\tilde{f}"'] &  & \LL \ar[dl,shift right=.80,"f"] \ar[dl,shift left=.80,"g"'] \\
	 & M
	\end{tikzcd}
\end{center}
We can therefore consider the coequalizer (taken in $Set^*$, i.e. after application of the forgetful functor) of morphisms $f$ and $g$
\begin{center}
	\begin{tikzcd}
	\LL \ar[r,shift left=.20,"f"] \ar[r,shift right=.80,swap,"g"] & M \ar[r,"q"]	& \frac{M}{\sim} \\
	\end{tikzcd}
\end{center}
with $(M/\sim \;, q)=coeq_{Set^*}(f,g)$. If $\Pi$ is the multivariable product on $M$, we want to show that $\Pi q$ does not factor via $q^*$ to a product of $M/\sim$. First of all note that
\begin{align*}
x=af \sim ag = xy; \\
x=bg \sim bf = y; \\
\end{align*}
Since $af=\hat{a}f$, $bf=\hat{b}f$ and similarly for $g$, we get that $\sim$ identifies exactly the non-identity elements of $M$, leaving us with $\displaystyle \frac{M}{\sim} \cong \frac{\mathbb{Z}}{2 \mathbb{Z}}$, which has a unique structure of partial group coinciding with that of the group. As words $(af,bf)$ and $(af,bg)$ have equal image under $q^*$, they represent the same word on $M /\sim$. However
\[
(af,bf)\Pi=(x,y)\Pi = xy,
\]
\[
(af,bg)\Pi = (x,x)\Pi = x^2=1,
\]
but $1 \not \sim xy$ in $M / \sim$.
\end{counterex}

\subsection{The category $Part$ is cocomplete} \label{cocompleteness}
The situation depicted by the above counterexample is analogous to that happening in the category of groups and group homomorphisms. When building coequalizers in $Grp$, we need to factor out not only those relations afforded by the morphisms in the diagram (namely $f$ and $g$), but the smallest \emph{normal subgroup} containing those relations. Similarly, when considering a coequalizer in $Part$, we need to quotient over a relation which allows the multivariable product to factor through the quotient itself. The difficulties of dealing with partial groups come from the lack of a substructure which controls this factorization process in the same way as normal subgroups control factor groups; therefore we need to identify the correct equivalence relation in a different way.\\
Let's consider any equivalence relation $\sim$ defined on the underlying set of a partial group $\LL'$; let $q: \LL' \fto \frac{\LL'}{\sim}$ be the canonical projection and define $D \left( \frac{\LL'}{\sim} \right)$ as in Section \ref{coequaliz partial grous}. Then the fibers of the map $q^*: D(\LL') \fto D \left( \frac{\LL'}{\sim} \right) $ induce an equivalence relation $\equiv$ on $D(\LL')$ and clearly $q^*$ is surjective being $q$ surjective. For $u=(u_i)_i, \, v=(v_i)_i \in D(\LL')$, the relation $\equiv$ is given by
\begin{equation*}
\label{first eq rel charact}
\tag{1}
u \equiv v \iff u,v \in (u)(q^*)\inv \iff u_i \sim v_i \; \forall i,
\end{equation*}
in particular $u$ and $v$ have equal length. Hence we get
\begin{center}
	\begin{tikzcd}
	D(\LL') \ar[r,"q^*"] \ar[d,"\Pi"] & D \left( \frac{\LL'}{\sim} \right) = \frac{D(\LL')}{\equiv} \ar[d,dashed] \\
	\LL' \ar[r,"q"] & \frac{\LL'}{\sim}
	\end{tikzcd}
\end{center}
The map $\Pi q$ is also surjective, so we can consider the equivalence relation $\approx$ on $D(\LL')$ afforded by fibers of $\Pi q$ and obtain $\frac{\LL'}{\sim} \cong \frac{D(\LL')}{\approx}$. The square diagram above then becomes
\begin{center}
	\begin{tikzcd}
	D(\LL') \ar[r,"q^*"] \ar[dr,"\Pi q"'] & D \left( \frac{\LL'}{\sim} \right) = \frac{D(\LL')}{\equiv} \ar[d,dashed] \\
	 & \frac{\LL'}{\sim} = \frac{D(\LL')}{\approx}
	\end{tikzcd}
\end{center}
so that $\Pi q$ factors through $q^*$ if and only if  $\, \equiv \, \subseteq \, \approx$. \\
We are then led to search for an equivalence relation $\sim$ on $\LL'$ such that the induced relations $\equiv$ and $\approx$ satisfy $\, \equiv \, \subseteq \, \approx$; we would then consider such a smallest equivalence relation as a candidate for our colimit. Let's also describe $\approx$ in terms of $\sim$;
\begin{equation*}
\label{second eq rel charact}
\tag{2}
u \approx v \iff (u)\Pi q = (v) \Pi q \iff u\Pi \sim v\Pi.
\end{equation*}
Combining \eqref{first eq rel charact} and \eqref{second eq rel charact}, we need to find $\sim$ such that, for $u=(u_i)_i, v=(v_i)_i \in D(\LL')$:
\begin{equation*}
\label{condition of sim}
\tag{*}
u_i \sim v_i \, \forall i \implies u\Pi \sim v\Pi.
\end{equation*}
Let now $\sim_0$ be the equivalence relation on $\LL'$ generated by the pairs of the form $(xf, xg)$ for some $x \in \LL$ (this is the equivalence relation affording the coequalizer in $Set$); clearly the set
\[
E :=\{ \Re \subseteq \LL' \times \LL' \mid \Re \text{ is an eq. rel. containing $\sim_0$ and satisfying \eqref{condition of sim}} \}
\]
is not empty, since it contains $\LL' \times \LL'$. Since inclusion of $\sim_0$ and property $\eqref{condition of sim}$ are preserved by intersections, there exists the smallest equivalence relation we are searching for, namely
\[
\R = \bigcap_{\Re \in E} \Re.
\]

\begin{lemma}
	\label{Part has coeq}
	The category $Part$ has all coequalizers.
\begin{proof}
	Consider morphisms between partial groups $f,g: \LL \fto \LL'$, where $\LL=(\LL,D(\LL),\Pi_\LL,i_\LL)$ and $\LL'=(\LL', D(\LL'), \Pi, i)$; build the setwise coequalizer $\frac{\LL'}{\sim}$ and the relation $\R$ (which depends on $\sim$) on $\LL'$ as above. Consider the projection $t: \LL' \fto \frac{\LL'}{\R}$ and the diagram
	\begin{center}
		\begin{tikzcd}
		\LL \ar[r,shift left=1.2,"f"] \ar[r,shift right=.2,swap, "g"] & \LL' \ar[r,"t"] & \frac{\LL'}{\R}
		\end{tikzcd}
	\end{center}
Then we prove that: 
\begin{enumerate}
	\item $\LL'$ induces a structure of partial group on $\frac{\LL'}{\R}$;
	\item $t$ is a morphism of partial groups and $ft=gt$;
	\item $\left( \frac{\LL'}{\R} , t \right) = coeq_{Part}(f,g)$.
\end{enumerate}
\begin{enumerate}
	\item Define $D:=D(\LL' / \R)$ as in \ref{definition:of domain}, that is as the set of images of elements of $D(\LL')$ via the map $t^*: W(\LL') \fto W(\LL' / \R)$. By construction $\R$ is the smallest equivalence relation on $\LL'$ such that $\Pi t$ factors as in the following commutative diagram:
	\begin{equation*}
	\label{diagram-defining-pi} \tag{$\top$}
	\begin{tikzcd}
	D(\LL') \ar[r,"t^*"] \ar[d,"\Pi"] & D \left( \frac{\LL'}{\R} \right) \ar[d,dashed, "\pi"] \\
	\LL' \ar[r,"t"] & \frac{\LL'}{\R}
	\end{tikzcd}
	\end{equation*}
	For any $x \in \LL'$ set $[x] \in \LL'/\R$ to be the class with representative $x$ and let $\pi$ be the (unique) map that realizes such factorization, that is $\Pi t = t^* \pi$. This means that for any word $\ov{u}=([u_1],\dots,[u_n]) \in D$ the representatives $u_1, \dots, u_n$ may be chosen so that $(u_1,\dots,u_n)\in D(\LL)$ and then we have $(\ov{u})\pi = (u_1,\dots,u_n)\Pi t$.\\
	If $i$ is the inversion map on $\LL'$, define an inversion map $j$ on $\LL'/\R$ by $([x] )j := [(x)i]$. Since $\R$ is by definition the smallest equivalence relation containing $\sim$ and such that \eqref{condition of sim} holds, then it is generated by the pairs $(x,y) \in \LL' \times \LL'$ with either $x=zf$ and $y=zg$ for some $z \in \LL$ or $x=u\Pi$ and $y=v\Pi$ for some words $u=(u_k)_k, \, v=(v_k)_k \in D(\LL')$ of equal length with $u_k \sim v_k$ for every $k$. By Lemma \ref{propos:factor on relation} we only need to check that $j$ is constant over such pairs. In the first case, since $f$ and $g$ are morphisms of partial groups we get $(x)i = ((z)i_\LL)f$ and $(y)i=((z)i_\LL)g$, so $((x)i , (y)i) \in \R$; in the other case, since inverting $(u_k)$ and $(v_k)$ consists in taking inverses in opposite order, we have $(u\Pi)i = ((u)i)\Pi \enspace \R \enspace ((v)i)\Pi = (v\Pi)i$. Thus $j$ is well defined. \\
	We need to prove that $(\LL'/\R, D,\pi,j)$ is a partial group. Axiom (1) of Definition \ref{def partial gr} is an immediate consequence of the definition of $D$. For any $x \in \LL'$ we have $([x])\pi = [(x)\Pi] = [x]$, which is axiom (2). Suppose now to have a word $\ov u \circ \ov v \circ \ov w \in D$; then there exists a lift of it, say $u \circ v \circ w \in D(\LL')$, via $t^*$. As axiom (3) holds for $\LL'$, $u \circ v\Pi \circ w \in D(\LL')$ and $(u \circ v\Pi \circ w)\Pi = (u \circ v \circ w)\Pi$; but we also have $(u \circ v\Pi \circ w)t^* = \ov u \circ \ov v \pi \circ \ov w$, so this is an element of $D$ and $(\ov u \circ \ov v \pi \circ \ov w)\pi = (u \circ v\Pi \circ w)t^* \pi = (u \circ v\Pi \circ w)\Pi t = (u \circ v \circ w)\Pi t = (u \circ v \circ w)t^* \pi = (\ov u \circ \ov v \circ \ov w)\pi$. This is axiom (3) for $\LL'/\R$. Let now $\ov u \in D$ and $u \in D(\LL')$ be a lift of $\ov u$ via $t^*$, then $(u)i \circ u \in D(\LL')$ and $((u)i \circ u)\Pi = 1_{\LL'}$, yielding $(\ov u)j \circ \ov u = (u)it^* \circ (u)t^* = ((u)i \circ u)t^* \in D$ and $((\ov u)j \circ \ov u)\pi = (((u)i \circ u)t^*)\pi = ((u)i \circ u)\Pi t = (1_{\LL'})t = 1_{\LL'/\R}$, therefore, showing that also axiom (4) holds.
	\item By definition of $D$ and $t$, we have $(D(\LL'))t^* = D$; moreover the product $\pi$ is defined so that the square (\ref{diagram-defining-pi}) is commutative. This shows that $t$ is a morphism of partial groups. As $\sim \, \subseteq \, \R$, we also have $ft=gt$.
	\item If $\tau: \LL' \fto \M=(\M,D(\M),\Pi_\M,i_\M)$ is a partial group homomorphism such that $f\tau= g\tau$, define $\psi$ as in the diagram below by:
	\[
	\psi: [x] \mapsto (x)\tau
	\]
	\begin{center}
		\begin{tikzcd}
		\LL \ar[r,shift left=1.2,"f"] \ar[r,shift right=.2,swap, "g"] & \LL' \ar[r,"t"] \ar[d,"\tau"] & \frac{\LL'}{\R} \ar[dl, dashed, "\psi"] \\
				& \M
		\end{tikzcd}
	\end{center}
	If $\psi$ is well defined, then it is clearly unique such that the diagram above commutes. Consider now the following diagram, which is commutative since $\tau$ is a morphism of partial groups.
	\begin{equation*}
		\begin{tikzcd}
		D(\LL') \ar[r,"\tau^*"] \ar[d,"\Pi"] & D (\M) \ar[d, "\Pi_\M"] \\
		\LL' \ar[r,"\tau"] & \M
		\end{tikzcd}
		\end{equation*}
	Let $\R_{\tau}$ be the equivalence relation induced by the fibers of $\tau$ on $\LL'$; it induces equivalence relations $\equiv_{\tau}$ and $\approx_{\tau}$ on $D(\LL')$ via fibers of, respectively, $\tau^*$ and $\Pi \tau$. The commutativity of the diagram above yields the fact that $\R_{\tau}$ fulfills \eqref{condition of sim}. Since $f\tau=g\tau$ we also have that $\R_\tau$ contains all the pairs $(xf,xg)$ with $x \in \LL$. $\R$ being the smallest equivalence relation on $\LL'$ satisfying these two properties, we have $\R \subseteq \R_{\tau}$, which means exactly that $\psi$ is well defined.
\end{enumerate}
\end{proof}
\end{lemma}

We finally obtain
\begin{theorem}
	\label{cocompl-of-part}
	The category $Part$ is cocomplete.
	\begin{proof}
		Combine Theorems \ref{part has coproducts}, \ref{Part has coeq} and Proposition \ref{prop:co complete cat}.
	\end{proof}
\end{theorem}

\begin{example}
\label{example: follow-up of counterex.}
We now compute the coequalizer $coeq(f,g)$ for the partial groups and maps defined in Counterexample \ref{counterex to factoriz}; in this specific case this computation is particularly easy.\\
Recall that $\sim$ is the equivalence relation on $M$ such that $M/ \sim$, paired with the projection map, is the coequalizer of $f$ and $g$ in $Set$; we have already seen that $|M / \sim|=2$ and also that $M  / \sim$ cannot be the coequalizer in $Part$. By Lemma \ref{Part has coeq} there exists an equivalence relation $\R$ on $M$ such that $M/\R$ is the coequalizer of $f$ and $g$ in the category of partial groups; in particular, such $\R$ satisfies $\sim \subseteq \R$. Then we get $\sim \subsetneqq \R$, which implies $|M/ \R|=1$; therefore $coeq_{Part}(f,g) = M / \R = \{1\}$, the trivial (partial) group, paired as usual with the projection over the quotient.
\end{example}

\section{Limits in $Part$}
\label{section:limits}

The situation regarding limits of partial groups is definitely simpler and A. Chermak has already proven completeness of the category $Part$ in \cite[Appendix A]{Chermak:2022}. For the purpose of collecting all results in one single paper, we nonetheless provide a proof here. The reader who is already familiar with Chermak's proof may skip this section.\\
Let's consider for a moment groups; there are a free-construction functor $F: Set \fto Grp$ and a forgetful functor $U: Grp \fto Set$ such that $F \dashv U$, so in particular $U$ preserves limits, as it is a right adjoint, . The situation with partial groups is the same: we have a free-construction functor $F: Set^* \fto Part$ and a forgetful functor $U: Part \fto Set^*$ such that $F \dashv U$. Hence a limit in $Part$ must have as underlying set the pointed setwise limit.\\

We now sketch a proof of existence of limits in $Part$.

\begin{theorem} \label{complete}
	The category Part is complete.
\begin{proof}
	Again we prove that $Part$ has all products and equalizers. \\
	
	\emph{Step 1.} If $\{ \LL _a \}_{a \in I}$ is a family of partial groups, consider the setwise cartesian product $\LL = \bigtimes \LL_a$; the projections $p_a: \LL \rightarrow \LL_a$ extend in the usual way to $p_a^*: W(\LL) \rightarrow W(\LL_a)$. Define $D(\LL) = D$ as
	\[
	D:= \{ w=( (w_a^1), \dots, (w_a^n) )_a \in W(\LL) \mid \forall a \in I, wp_a^*=(w_a^1, \dots, w_a^n) \in D(\LL_a) \}.
	\]
	and a product $\Pi$ on $\LL$ by $(w)\Pi:= ((wp_a^*)\Pi_a)_a$. The inversion map $i$ on $\LL$ is trivially defined componentwise; it is a straightforward exercise to check that $(\LL, D, \Pi, i)$ is a partial group. Considering the diagram below
	
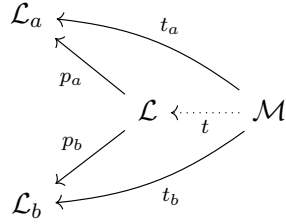
\begin{center}
	\begin{tikzcd}
		\LL_a \\
			& \LL \ar[lu, "p_a"] \ar[ld, swap, "p_b"] & \mathcal{M} \ar[l, dotted, "t"] \ar[llu,bend right=15, swap, "t_a"] \ar[lld,bend left=15, "t_b"] \\
		\LL_b
	\end{tikzcd}
	\end{center}
$t$ is defined by $t: m \mapsto (mt_a)_a$. Notice that, for $w=(m_1, \dots, m_n) \in D(\mathcal{M})$, we have $wt_a^*=(m_1t_a, \dots, m_nt_a) \in D(\LL_a)$ for all $a \in I$, so $wt \in D$. Clearly $t$ is a morphism of partial groups, unique satisfying the universal property of products.\\

\emph{Step 2.} Consider morphisms of partial groups $f,g: \LL \rightarrow \LL'$; define 
\[
\mathcal{E}:= \{ x \in \LL \mid xf=xg \}.
\]
We have
\begin{center}
\begin{tikzcd}
	\mathcal{E} \ar[r,hook,"j"] & \LL \ar[r,shift left=.75, "f"] \ar[r, shift right=.75, swap, "g"] & \LL';
\end{tikzcd}
\end{center}
it is a simple exercise to prove that $\mathcal{E}$ is a partial group, $j$ a partial groups morphism and that the pair $(E,j)= eq(f,g)$.\\

As a consequence of Proposition \ref{prop:co complete cat} we obtain that $Part$ is complete. In addition, it is now a trivial observation the fact that the full subcategory of $Part$ whose objects are the finite partial groups is finitely complete.
\end{proof}
\end{theorem}

\begin{proof}[Proof of Theorem \ref{theoremA}]
This is obtained by combining Theorems \ref{cocompl-of-part} and \ref{complete} and the remark after Lemma \ref{part has coproducts}.
\end{proof}

\subsection{Partial groups as simplicial sets}
It has been initially conjectured by C. Broto and then formalized in joint work with A. Gonz\'{a}lez that partial groups carry a natural structure of simplicial sets. Indeed it turns out that $Part$ is equivalent to a full subcategory of the category $sSet$ of simplicial sets, as described in \cite{Gonzalez:2015}. We wish to show that such subcategory is not closed under formation of colimits in $sSet$, so the well-understood structure of colimits of simplicial sets cannot be used for computing colimits of partial groups. We begin with recalling some definitions and elementary results; more details may be found in \cite{Goerss/Jardine:1999a}.\\
Let $\Delta$ be the category of finite partially ordered sets with monotone, non-decreasing functions as morphisms. It is possible to show that $\Delta$ is equivalent to the category whose objects are the finite ordered sequences $[n]=\{0<1<\dots <n\}$ and whose morphisms are all non-decreasing functions. This is known as the skeletal subcategory; in the following we will restrict ourselves to this subcategory and denote it as well by $\Delta$.
\begin{definition}
	A simplicial set is a functor
	\[
	X: \Delta ^{op} \fto Set;
	\]
	hence we may think at a simplicial set as a sequence of sets $X_n := [n]X$ with morphisms induced by the monotone functions in $\Delta$ taken with opposite direction.\\
	The elements of $X_0$ are called \emph{vertices} of the simplicial set $X$; if $X_0 = \{v\}$ (i.e. $X$ has a single vertex), then $X$ is said to be \emph{reduced}.
\end{definition}
Simplicial sets together with natural transformations form a category, usually denoted by $sSet$. The structure of colimits of simplicial sets is well known; if $T:\C \fto sSet$ is a diagram over the small category $\C$, then
\[
colim(T) : \Delta^{op} \fto Set
\]
is the simplicial set with $colim(T)_n = \underset{c \in \C}{colim} \; [n](cT)$ and suitable morphism induced by the universal property of colimits (we will not need to study in detail the morphisms).\\
In \cite[Theorem 4.8]{Gonzalez:2015} Gonz\'{a}lez proves that partial groups, when regarded as simplicial sets, are characterized by being reduced, among other properties. Therefore, if $\LL$ and $\M$ are partial groups, viewed as simplicial sets they are reduced, i.e. $\LL_0=\{ \star \}$ and $\M_0=\{ * \}$; however the coproduct computed in $sSet$ is not reduced, since $(\LL \sqcup \M)_0 = \LL_0 \sqcup \M_0 = \{ \star, * \}$, thus it cannot be a partial group.

\section{Quotients of partial groups}
\label{section:quotients}

It is common to control algebraic structures by means of a generating set and a set of relations. The generating process is normally related to the existence of free objects, whereas considering relations means looking at colimits. We now want to investigate this situation for partial groups. First of all we need to define quotients of partial groups; such a notion was previously defined only for specific subclasses, namely for groups and for the so-called \emph{localities} (see \cite[Definition 2.7 and Sections 3 and 4]{Chermak:2022}); in order to define quotients for all partial groups we will take advantage of the categorical point of view, generalizing what happens with groups and group homomorphisms. We will then show that our definition coincides with that given by A. Chermak for localities.\\
We start with defining substructures; the names assigned to them follow the notation used by Chermak in \cite{Chermak:2022} and \cite{Chermak:2017:III}.

\begin{definition}
	Let $\LL = (\LL, D, \Pi, i)$ be a partial group and $\H \subseteq \LL$ a non-empty subset. For any $g \in \LL$ write $D(g) :=\{ x \in \LL \mid (g\inv,x,g) \in D)\}$.
	\begin{itemize}
		\item[(a)] The quadruple $(\H, E_\H, \Pi_{E_\H}, i_\H)$ is an \emph{impartial subgroup} of $\LL$ if $E_\H \subseteq D \cap W(\H)$, $\Pi_{E_\H}$ and $i_\H$ are the restrictions of, respectively, $\Pi$ to $E_\H$ and $i$ to $\H$ and if $(\H, E_\H, \Pi_{E_\H}, i_\H)$ is a partial group.
		\item[(b)] An impartial subgroup $\H=(\H, E_\H, \Pi_{E_\H}, i_\H)$ of $\LL$ is a \emph{partial subgroup} if $E_\H = D \cap W(\H)$. In this case $E_\H$ is completely determined by the subset $\H$ and the partial group $\LL$.
		\item[(c)] $\H$ is a \emph{partial normal subgroup} of $\LL$ if it is a partial subgroup of $\LL$ and, for every $x \in \H$ and $g \in \LL$ such that $x \in D(g)$, $(g\inv, x, g) \Pi \in \H$  holds.
	\end{itemize}
\end{definition}

Note that $(\H, E_\H, \Pi_{E_\H}, i_\H)$ being a partial group in (a) implies that $E_\H$ is closed under application of $\Pi_{E_\H}$ and also $\H$ is closed under $i_\H$. The next results will make clear that impartial subgroups are the correct substructure to consider from the categorical point view; however it is often better to work with partial subgroups since they have better properties. According to Chermak's definition in \cite{Chermak:2017:III}
\begin{itemize}
	\item[(a')] $\H$ is an impartial subgroup of $(\LL,D,\Pi,i)$ if $\H$ is the homomorphic image of a partial group homomorphism $f: \M \fto \LL$ for some partial group $(\M, D_\M,\Pi_\M,i_\M)$.
\end{itemize}
With regard to (a'), the homomorphic image $\mathcal{I}:=im(f)$ becomes a partial group $(\mathcal{I}, D_\mathcal{I}, \Pi_\mathcal{I}, i_\mathcal{I})$ by setting $D_\mathcal{I}:=(D_\M)f^*$ and defining $\Pi_\mathcal{I}: = \Pi |_\mathcal{I}$ and $i_\mathcal{I}:= i |_\mathcal{I}$, respectively the restrictions of $\Pi$ to $D_\mathcal{I}$ and of $i$ to $\mathcal{I}$. Details may be found in the proof of the following Lemma.

\begin{lemma}\label{equiv-def-impar-subgr}
The definitions (a) and (a') of an impartial subgroup are equivalent.
\end{lemma}
\begin{proof}
Note that, if $\H$ is an impartial subgroup of a partial group $(\LL, D, \Pi, i)$ according to (a), then the inclusion map $i: \H \fto \LL$ is a morphism of partial groups; thus (a) implies (a').\\
Conversely, if $\H = im(f)$ for some partial group $(\M,D_\M,\Pi_\M,i_\M)$ and homomorphism $f:\M \fto \LL$, then we have $D_\H:=(D_\M)f^* \subseteq W(\H) \cap D_\LL$. Taking restrictions $\Pi_\H$ of $\Pi$ to $D_\H$ and $i_\H$ of $i$ to $\H$, we need to show that $(\H, D_\H, \Pi_\H, i_\H)$ is a partial group. Since $\M \subseteq D_\M$, we have $\H=(\M)f \subseteq D_\H$; moreover suppose that $u=(u_1,\dots,u_k), \, v=(v_1,\dots,v_l) \in W(\H)$ are such that $u \circ v \in D_\H$, then there exists $m = (m_1, \dots, m_k, n_1, \dots, n_l) \in D_\M$ such that $u \circ v = (m)f^*$. In particular, $m':=(m_1, \dots, m_k), n':=(n_1,\dots, n_l) \in D_\M$ and $u=(m')f^*, \, v=(n')f^*$, so that $u,v \in D_\H$. This proves axiom (1) of Definition \ref{def partial gr}. Axiom (2) is trivial, since the product is defined by restricting that of $\LL$. Let now $u \circ v \circ w \in D_\H$ with $u=(u_1,\dots,u_k), \, v=(v_1,\dots,v_l)$ and $w=(w_1,\dots,w_r)$; by axiom (1) $u,v,w \in D_\H$ and there exist $m':=(m_1, \dots, m_k), \, n':=(n_1,\dots, n_l), \, t':=(t_1,\dots,t_r) \in D_\M$ so that $m'\circ n' \circ t' \in D_\M$ and $u=(m')f^*, \, v=(n')f^*, \, w=(t')f^*$. Since $m' \circ (n')\Pi_\M \circ t' \in D_\M$, we have $u \circ (v)\Pi_\H \circ w = (m')f^* \circ (n')f^* \Pi \circ (t')f^* = (m' \circ (n')\Pi_\M \circ t')f^* \in D_\H$. The equality $(u \circ (v)\Pi_\H \circ w)\Pi_\H = (u \circ v \circ w)\Pi_\H$ follows at once, recalling that these are, in particular, elements of $D$ and $\Pi_\H$ is the restriction of $\Pi$. Let now $u$ and $m'$ be as above; then $(m')i_\M \circ m' \in D_\M$ and it has product $1_\M$, thus $((m')i_\M \circ m')f^* = (u)i_\H \circ u \in D_\H$ and $((u)i_\H \circ u)\Pi_\H = ((m')i_\M \circ m')f^* \Pi = ((m')i_\M \circ m')\Pi_\M f = (1_\M)f = 1_\LL = 1_\H$. This completes the proof.
\end{proof}
In particular, the impartial subgroups of a partial group $\LL$ are, up to isomorphism, exactly the partial groups embedded in $\LL$ and play the same role of subgroups of a group; for example, we will show that they provide a factorization of a morphism in an epimorphism followed by a monomorphism (the inclusion). \\
With respect to quotients, we look at what happens with groups from the categorical point of view. If $G$ is a group and $H$ a subgroup of $G$, it is well known that $H$ affords a quotient structure if and only if it is normal. Nevertheless, it is possible to \emph{kill} any subgroup $H$ through a quotient at the price of replacing $H$ with its the normal closure, that is $N:= \langle H^G \rangle$. Since the category $Grp$ is cocomplete and has a $0$-object, namely the trivial group, one can build the trivial map $\hat{1}$ between any two groups $G$ and $H$ by passing through the trivial group,
\begin{center}
	\begin{tikzcd}
	H \ar[rr,"\hat{1}"] \ar[rd, "\hat{1}"] & & G \\
	& 1 \ar[ru, "\hat{1}"]
	\end{tikzcd}
\end{center}
In particular, if $H \le G$ and $i: H \fto G$ is the inclusion map, then it is a straightforward computation checking that $coker(i) = coeq(i,\hat{1}) = G/N$, where $N:= \langle H^G \rangle$, so the following is a coequalizer diagram
\begin{center}
	\begin{tikzcd}
	H \ar[r, hook, shift left=1.2,"i"] \ar[r,shift right=.2,swap, "\hat{1}"] & G \ar[r,"q"] & \frac{G}{N}
	\end{tikzcd}
\end{center}
Thus it is reasonable to generalize the notion of a quotient to any subgroup $H$ by setting  $G/ H := coeq_{Grp}(i,\hat{1}) = coker(i) = G/ \langle H^G \rangle$, $i : H \fto G$ being the inclusion map and $\hat{1}$ the trivial map. This line of reasoning extends directly to partial groups since also $Part$ has all limits and colimits and is pointed, the trivial group $1$ being the $0$-object. Denoting the trivial map always by $\hat{1}: \LL \fto \M$, for every morphism $f: \LL \fto \M$ one then has $ker(f) = eq(f,\hat{1})$ and $coker(f) = coeq(f,\hat{1})$.
It is immediate to verify that our notion of kernel of a morphism coincides with the one given by Chermak in \cite[Definition 1.11]{Chermak:2022}; we are then led to the following definition.
\begin{definition}\label{defin-quotient-local}
	Let $\LL$ be a partial group and $\H \subseteq \LL$ be an impartial subgroup. Let $j : \H \fto \LL$ be the inclusion morphism. Define the quotient partial group $\LL / \H$ by
	\[
	\LL / \H := coker(j)=coeq_{Part}(j,\hat{1}),
	\]
	where $\hat{1}$ is the trivial map from $\H$ to $\LL$.
\end{definition}
In particular we can factor any morphism of partial groups $f: \LL \fto \M$ as an epi followed by a mono, that is
\begin{center}
\begin{tikzcd}
f: &[-25pt] \LL \ar[rr]  \ar[rd,two heads] & & \M \\ [-20pt]
	& 	& im(f) \ar[ru,hook]
\end{tikzcd}
\end{center}
and it is a simple exercise to see that $coker(f) = \M / im(f)$.\\

It is clear that, when $\H$ and $\LL$ are groups, the definition above coincides with the known quotient in $Grp$. We now show that it also coincides with the definition of a \emph{quotient locality} given by Chermak in \cite{Chermak:2022}. Chermak's construction applies over a finite locality $\LL$ (see \cite[Definition 2.7]{Chermak:2022}) and a partial normal subgroup $\N \trianglelefteq \LL$; in \cite[Section 3]{Chermak:2022} the author defines an equivalence relation on $\LL$ by identifying the maximal classes (named the \emph{cosets} of $\N$) of a reflexive and transitive relation. In particular, \cite[Proposition 3.14]{Chermak:2022} shows that cosets are all of the form $\N f$ for some element $f \in \LL$.

\begin{notation}
We follow the notation adopted in \cite[Definition 3.6 and Corollary 4.5]{Chermak:2022}. Thus $\LL = (\LL, \Delta, S)$ is a finite locality and $\N \le \LL$ a partial normal subgroup, $\overline{\LL} := \LL / \N$ is the set of $\N$-cosets of $\LL$ and $\rho: \LL \fto \overline{\LL}$ the projection map; just for the following lemma the notation $\LL / \N$ is used exclusively for Chermak's construction. Setting $\overline{\Delta} := \{(P)\rho =: \overline{P} \mid P \in \Delta \}$, \cite[Lemma 3.16 and Theorem 4.3]{Chermak:2022} show that $(\overline{\LL}, \overline{\Delta}, \overline{S})$ is a locality.
\end{notation}

\begin{lemma}
Following the notation above, if $\LL = (\LL, \Delta, S)$ is a finite locality, $\N \le \LL$ a partial normal subgroup of $\LL$ and $\rho : \LL \fto \LL / \N$ the projection sending each element $g \in \LL$ to the unique maximal coset of $\N$ containing $g$, then the locality $(\ov{\LL} = \LL / \N, \ov{\Delta}, \ov{S})$ is the quotient $coker(\N \hookrightarrow \LL)$ as in Definition \ref{defin-quotient-local}.
\end{lemma}
\begin{proof}
	Recall that also $coker(\N \hookrightarrow \LL)$ is defined by setting an equivalence relation on $\LL$. Since the domains, product maps and inversion maps of $\LL / \N$ and $coker(\N \hookrightarrow \LL)$ are in both cases induced by those of $\LL$ (\cite[Lemma 3.16]{Chermak:2022}), all we need to show is that $\LL / \N = coker(\N \hookrightarrow \LL)$ as sets (that is, the equivalence relations are the same) and $\rho$ is the projection map also in the coequalizer diagram defining $coker(\N \hookrightarrow \LL)$. Let $\approx$ be the equivalence relation on $\LL$ with classes $\LL / \N$ and $\R$ be the one defining $coker(\N \hookrightarrow \LL)$; since the product of $\LL$ factors over $\LL / \N$, by definition of $\R$ we have $\R \subseteq \approx$.\\
	Suppose now that $\mathcal{E}$ is an equivalence class of $\approx$; by \cite[Proposition 3.14]{Chermak:2022} there exists an element $f \in \LL$ such that $\mathcal{E} = \N f$. Let $g \in \N f$, so $g=nf=(n,f)\Pi$ for some $n \in \N$; since the words $(n,f)$ and $(1,f)$ are clearly componentwise $\R$-related, by the definition of $\R$ we also have $g =(n,f) \Pi \; \R \; (1,f)\Pi = f$. This shows that $ \mathcal{E}=\N f$ is contained in an equivalence class of $\R$. Thus $ \approx \subseteq \R$, so equality holds.\\
	That $\rho$ is also the quotient map over $coker(\N \hookrightarrow \LL)$ is now trivial.
\end{proof}

We end this section showing that, even though impartial subgroups are the natural categorical subojects in $Part$, they are after all redundant for the purpose of computing quotients.

\begin{lemma}
\label{lemma:quotients-part-impart-sbgr}
Let $\LL=(\LL,D,\Pi,i)$ be a partial group, $\N=(\N,E,\Pi_\N, i_\N)$ be an impartial subgroup of $\LL$ and $\M=(\M,D_\M, \Pi_\M, i_\M)$ be the partial subgroup of $\LL$ generated by the subset $\N$ in the sense of \cite[Lemma 1.8(e) and 1.9]{Chermak:2022}. Then $\LL/\N \cong \LL/\M$.
\begin{proof}
We have an obvious inclusion morphism $i:\N \hookrightarrow \M$, so we may consider the following diagram
\begin{center}
\begin{tikzcd}
\N \ar[r,hook,"i"]	& \M \ar[r,hook,"i_\M"] 	& \LL \ar[d,two heads, "q_\N"] \ar[r,two heads, "q_\M"]	& \LL/\M  \ar[dl,shift right=1, "\psi"']\\
				&					& \LL/\N \ar[ur,shift right = 1, "\varphi"']
\end{tikzcd}
\end{center}
with $i i_\M = i_\N$, the inclusion morphism of $\N$ in $\LL$. Since $i_\N q_\M = i i_\M q_\M = \hat{1}$, by the universal property of the cokernel $(\LL/\N,q_\N)$ there exists a unique map $\varphi: \LL/\N \fto \LL/\M$ such that $q_\M = q_\N \varphi$. Now \cite[Lemma 1.9]{Chermak:2022} shows that $\M = \bigcup_{i \ge 0} \N_i$, with $\N_0:=\N$ and $\N_i := \{ (n)\Pi \mid n \in W(\N_{i-1}) \cap D \}$; if $m \in \N_0$, then $(m)q_\N = 1_{\LL/\N}$. By induction, let $i \ge 1$ and $m \in \N_{i}$, then $m=(m')\Pi$ for some $m' \in W(\N_{i-1}) \cap D$ and $(m)q_\N = ((m')\Pi)q_\N = ((m')q_\N^*)\Pi_{\LL / \N} = (1_{\LL / \N}, \dots, 1_{\LL / \N}) \Pi_{\LL / \N} = 1_{\LL / \N}$. This shows that $i_\M q_\N = \hat{1}$, so the universal property of $\LL/\M$ yields the existence of a unique morphism $\psi: \LL / \M \fto \LL / \N$ such that $q_\M \psi = q_\N$.\\
Now $q_\N \varphi \psi = q_\M \psi = q_\N$ and $q_\M \psi \varphi = q_\N \varphi = q_\M$, so the uniqueness property yields $\varphi \psi = id_{\LL/\N}$ and $\psi \varphi = id_{\LL/\M}$, hence $\LL/\N \cong \LL/\M$.
\end{proof}
\end{lemma}

\section{Generators and relations for partial groups}
\label{section:gen and relations}

This section is mostly dedicated to the proof of Proposition \ref{prop:every-pg-is-quot-free}. First of all, we shall address the fact that the behavior of partial groups can be counterintuitive; for example, in \cite[Section 5]{Chermak:2013} and \cite[Chapter 3]{Chermak:2016:II} Chermak describes a way of expanding a given locality $\LL$, the result being a new locality $\LL^+$ having the following properties: $\LL$ is an impartial subgroup of $\LL^+$ and $\LL^+$ is generated by the set $\LL$. In particular, whenever $\LL \subset \LL^+$ we get a proper embedding between two partial groups sharing an equal set of generators, namely $\LL$. This happens since the set of generated elements depends not only on the generators, but also on the product or, better, on the domains $D(\LL)$ and $ D(\LL^+)$. Another important aspect to keep in mind is the following: the free partial group over a set $X$ is very small if compared to the free group over $X$; for example, if $X$ is finite, the free partial group over $X$ is also finite, however the only finite free group is the trivial group, free over the empty set. This is again a consequence of the structure of the domain of a free partial group; being such domain very small, the product is defined only over few words.\\
The idea now consists in providing additional information controlling the size of the domain; this information will correspond to defining a forgetful functor $U$ from $Part$ to a suitable category which remembers the domain $D(\LL)$ of a given partial group $(\LL, D(\LL), \Pi, i)$. It turns out that such a functor must remember also the inversion map $i$ in order to obtain good categorical properties. We start with the following definition.\\

\begin{definition}
	For any set $X$, denote as usual by $W(X)$ the free monoid over $X$. Define the category $Set^s$ in the following way.
	\begin{itemize}
		\item The objects of $Set^s$ are triples $(X,S_X,i_X)$, where $X$ is a set, $S_X \subseteq W(X)$ and $i_X : X \fto X$ is an involutory bijection.
		\item A morphism $h: (X,S_X,i_X) \fto (Y,S_Y,i_Y)$ between two objects is given by a set-wise map $h: X \fto Y$ that satisfies the property $(S_X)h^* \subseteq S_Y$ and makes the following diagram commutative.
		\begin{equation*}
		\label{commutativity-morpphism-inversion maps}
		\tag{$\ddagger$}
		\begin{tikzcd}
		X \ar[r,"h"] \ar[d,"i_X"] &	Y \ar[d,"i_Y"] \\
		X \ar[r,"h"] &	Y
		\end{tikzcd}
		\end{equation*}
		\item Composition of morphisms is the usual composition of maps.
	\end{itemize}
\end{definition}
The set $S_X$ should be thought of as the set of words over which we wish to have a product defined; we then define a forgetful functor with the following shape:
\[
U: Part \fto Set^s, \qquad (\LL,D,\Pi,i)U=(\LL,D,i)
\]
Indeed, whenever we have a morphism of partial groups $f: \LL \fto \LL'$, by definition we have $(D)f^* \subseteq D'$ where $D$ and $D'$ are the domains of $\LL$, respectively $\LL'$, so that we can define $fU = f$; since $f$ respects the inversion map, the diagram \ref{commutativity-morpphism-inversion maps} is commutative. The first question to answer is whether this newly defined forgetful functor $U$ is a right adjoint, that is to say, whether it is possible to build free partial groups over objects of $Set^s$.

\subsection{Construction of free partial groups over $Set^s$}
\label{section:constr-free-part-gr}
Let $(X,S_X,i_X)$ be an object of the category $Set^s$, then we identify the free partial group over $(X,S_X,i_X)$ inside a suitable quotient of the free group with set of generators $X$. For any object $(X,S_X,i_X) \in Ob(Set^s)$ let $W :=W(X)$ be the free monoid on $X$ and $G(X)$ the free group on $X$. Set $G_X := G(X)/R(X)$, where $R(X)$ is the normal closure of the subgroup of $G(X)$ generated by the relations $(x)i_X \circ x = \emptyset = x \circ (x)i_X$ for all $x \in X$; we will often denote an element $g \in G_X$ simply by a representative $(x_1, x_2, \dots, x_n) \in G(X)$. If $x \in X$, denote by $\hat{x}$ its inverse in $G(X)$ (this notation will be used throughout the rest of the paper) and note that we get $R(X)\hat{x} = R(X)(x)i_X$ in $G_X$; then extend $i_X$ to an involutory bijection of $G(X)$ by $(y_1,\dots,y_n)i_X := ((y_n)i_X,\dots,(y_1)i_X)$ where $y_i \in \{x_i, \hat{x}_i\}$ for $x_i \in X$.

\begin{lemma}\label{lemma:description of G_X}
Let $X$ be a set, $i_X$ an involutory bijection on $X$ and $G(X)$ and $G_X$ be as above. Define the monoid homomorphism $t: W(X) \fto G(X)$ by $t: x \mapsto x$ for all $x \in X$. Then $i_X$ factors modulo $R(X)$ and affords the inversion map of the group $G_X$. Moreover
\begin{itemize}
\item[(a)] if $q: G(X) \fto G_X$ is the quotient map, then the composition $\mu:=tq: W(X) \fto G_X$ is a homomorphism of monoids;
\item[(b)] $\mu$ is surjective and $(1_{G_X})\mu \inv = R(X) \cap W(X)$; in particular, if $u,v \in W(X)$ are such that $(u)\mu = (v)\mu$, then $u \circ (v)i_X \in R(X) \cap W(X)$.
\end{itemize}
\end{lemma}
\begin{proof}
The subgroup $R(X)$ is $i_X$-invariant by definition; in $G_X$ we have $(R(X)(x_1, \dots, x_n))i_X = R(X)((x_1)i_X, \dots, (x_n)i_X) = R(X)(\hat{x_1}, \dots, \hat{x_n})$, showing that $i_X$ is the inversion map. Now (a) is trivial, as $q$ is a group homomorphism. Let $w=(w_1, \dots, w_n) \in G(X)$ be a representative of $R(X)w \in G_X$; then each $w_i$ is either of the form $x_i$ or of the form $\hat{x_i}$ for some $x_i \in X$. Define
\[
w'_i := \begin{cases*}
		w_i & \text{if $w_i \in X$}\\
		(w_i)i_X & \text{if $w_i \not \in X$}
		\end{cases*}
\]
and set $w':=(w'_1,\dots,w'_n)$; in particular $R(X)w_i = R(X) w'_i$ for every $i \in \{1,\dots,n\}$. Then $w' \in W(X)$ and $R(X)w=R(X)w'$, showing that $\mu$ is surjective.\\
Note that $t$ is injective, as $W(X)$ contains no non-trivial inverses of elements of $G(X)$, so $(1_{G_X})\mu \inv = R(X) \cap W(X)$ holds. Take now $u=(u_1,\dots,u_n), \, v=(v_1,\dots,v_m) \in W(X)$ with $(u)\mu = (v)\mu$, then we get
\[
(u_1,\dots,u_n)\circ (v_1,\dots,v_m)i_X \in R(X) \cap W(X)
\]
which is (b).
\end{proof}
The morphism $\mu$ induces also a map
\begin{center}
\begin{tikzcd}
\nu: S_X \ar[r] & W(G_X) \\[-25pt]
(x_1,\dots,x_n) \ar[r,mapsto] & ((x_1)\mu, \dots, (x_n)\mu) =(R(X)x_1, \dots, R(X)x_n),
\end{tikzcd}
\end{center}
which satisfies $\nu \Pi_{G_X} = \mu$; we will often write simply $S_X \subseteq W(G_X)$ in place of $(S_X)\nu \subseteq W(G_X)$.\\
We need the following lemma.
\begin{lemma}
\label{lemma:intersection of impartial subgr}
Let $(\LL,D, \Pi, i)$ be a partial group and $(\LL_k, D_k, \Pi_k, i_k)_{k \in K}$ a family of impartial subgroups of $\LL$ indexed over a set $K$. Set
\[
\M := \bigcap_{k \in K} \LL_k,
\]
then there is a unique structure of partial group over $\M$, namely $(\M, D_\M, \Pi_\M, i_\M)$, such that the following are satisfied:
\begin{itemize}
\item[(i)] $(\M, D_\M, \Pi_\M, i_\M)$ is an impartial subgroup of $\LL$,
\item[(ii)] if $\N=(\N,D_\N, \Pi_\N, i_\N)$ is an impartial subgroup of each of the $\LL_k$, then it is an impartial subgroup of $\M$.
\end{itemize}
We define $\M$ as the \emph{intersection} of the impartial subgroups $\LL_k$ in $\LL$.

\begin{proof}
Let $D_\M := \bigcap_{k \in K} D_k$, $\Pi_\M := \Pi \vert _{D_\M}$ and $i_\M := i \vert_{D_\M}$; 
we prove that $(\M, D_\M, \Pi_\M, i_\M)$ has the desired properties. Property (i) clearly follows once we show that $(\M, D_\M, \Pi_\M, i_\M)$ is a partial group. According to our definitions, $\M \subseteq D_k$ for every $k \in K$, so $\M \subseteq D_\M$. Let $u \circ v \in D_\M$, then $u \circ v \in D_k$ for every $k \in K$, so $u,v \in D_k$ for all $k$; this gives $u,v \in D_\M$ and axiom (1) of the definition of a partial group holds. Axiom (2) is trivial, since $\Pi_\M$ is defined as the restriction of $\Pi$. Suppose now that $u \circ v \circ w \in D_\M$, then $u \circ v  \circ w \in D_k$ for all $k \in K$, so 
\[
u \circ (v)\Pi_k \circ w \in D_k \quad \text{and} \quad (u \circ (v)\Pi_k \circ w) \Pi_k = (u \circ v \circ w) \Pi_k \qquad \forall k \in K
\]
Since $\Pi_\M = \Pi \vert _{D_\M} = \Pi_k \vert_{D_\M}$ for every $k \in K$, we have
\[
u \circ (v)\Pi_\M \circ w \in D_k \quad \text{and} \quad (u \circ (v)\Pi_\M \circ w) \Pi_\M = (u \circ v \circ w) \Pi_\M
\]
which is axiom (3). Let now $u \in D_\M$, then again $u \in D_k$ for every $k \in K$ and therefore
\[
(u)i_k \circ u \in D_k \quad \text{and} \quad ((u)i_k \circ u) \Pi_k = 1_\LL \qquad \forall k \in K.
\]
Thus, reasoning as for axiom (3) we get again
\[
(u)i_\M \circ u \in D_\M \quad \text{and} \quad ((u)i_\M \circ u) \Pi_\M = 1_\LL=1_\M
\]
which is axiom (4). This proves (i).\\
Suppose now that $\N$ is as in (ii), then for every $k \in K$
\[
D_\N \subseteq D_k, \quad \Pi_\N = \Pi_k \vert_\N = \Pi_\M \vert_{D_\N} \quad \text{and} \quad i_\N = i_k \vert_{D_\N} = i_\M \vert_{D_\N},
\]
which implies that the inclusion map $\N \hookrightarrow \M$ is a morphism of partial groups. This proves (ii).
\end{proof}
\end{lemma}

We follow the notation introduced in Section \ref{section:constr-free-part-gr}
\begin{lemma}\label{free-part-gr-str-set}
	Suppose that $(X,S_X,i_X) \in Set^s$, then there exist unique sets $\LL \subseteq G_X$ and $D \subseteq W(\LL) \subseteq W(G_X)$ such that:
	\begin{enumerate}
		\item $X \subseteq \LL$, $S_X \subseteq D$ and $(\LL)i_{G_X} \subseteq \LL$;
		\item denoting $\Pi:= \Pi_{G_X} \vert _{D}$ and $j=i_{G_X} \vert _{\LL}$, then $(D)\Pi \subseteq \LL$ and $(\LL, D, \Pi, j)$ is a partial group; in particular, for any word $s = (x_1,\dots,x_n) \in S_X$ we have $(s)\Pi \in \LL$;
		\item any other pair of sets $\LL' \subseteq G_X$, $D' \subseteq W(\LL')$ satisfying (1) and (2) is such that $(\LL,D,\Pi,j)$ is an impartial subgroup of $(\LL',D',\Pi',j')$, where $\Pi'$ and $j'$ are the restrictions of, respectively, $\Pi_{G_X}$ to $D'$ and $i_{G_X}$ to $\LL'$.
	\end{enumerate}
\end{lemma}

\begin{proof}
	We show that the group $G_X$ (so with $W(G_X)$ as domain) satisfies (1) and (2) and, afterwards, that properties (1) and (2) are stable under taking intersections of impartial subgroups.\\
	For $G_X$ we have $\Pi = \Pi_{G_X}$, $j=i_{G_X}$ and $(G_X,W(G_X),\Pi, j)$ clearly satisfies (1) and (2). Suppose now to have a family $\{\LL_k\}_{k \in K}$ of subsets $\LL_k \subseteq G_X$ with associated subsets $D_k \subseteq W(\LL_k)$, indexed over an index-set $K$, such that all pairs $(\LL_k, D_k)_{k \in K}$ satisfy (1) and (2) and set $\Pi_k := \Pi_{G_X} \vert_{D_k}$ and $j_k := i_{G_X} \vert_{\LL_k}$. Then by (2) each $(\LL_k, D_k, \Pi_k, j_k)$ is an impartial subgroup of $G_X$. Let $\M = (\M, D_\M, \Pi_\M, i_\M)$ be the intersection impartial subgroup of $G_X$, taken over all the $\{\LL_k\}_{k \in K}$, as defined in Lemma \ref{lemma:intersection of impartial subgr}. Clearly $X \subseteq \M$, $S_X \subseteq D_\M$ and $(\M)i_\M \subseteq \M$; as $(D_k)\Pi_k \subseteq \LL_k$ for every $k \in K$, we also have $(D_\M)\Pi_\M \subseteq \M$. Thus $\M$ satisfies properties (1) and (2).\\
	Let now $\{\LL_k\}_{k \in K}$ be the set of all impartial subgroups $(\LL_k, D_k, \Pi_k, j_k)$ of $G_X$ satisfying (1) and (2) and $(\LL, D, \Pi, j)$ be the intersection impartial subgroup of $G_X$ taken over the family $\{\LL_k\}$. Then it is a partial group satisfying (1) and (2) and, by definition, it satisfies also property (3).
\end{proof}
What we have proven is, actually, that there exists a unique smallest impartial subgroup of $G_X$ satisfying the properties (1) and (2).\\
Note that a morphism $f:(X,S_X,i_X) \fto (Y,S_Y,i_Y)$ in $Set^s$ induces a map $\tilde{f}: G_X \fto G_Y$ in the following way. First we build $f^+: G(X) \fto G(Y)$ by the universal property of the free group $G(X)$, as in the following diagram
\begin{center}
\begin{tikzcd}
X \ar[r,hook] \ar[d,"f"] & G(X) \ar[dl,dashed,"f^+"] \\
G(Y)
\end{tikzcd}
\end{center}
and note that $f^+$ restricts to the map $f^*: W(X) \fto W(Y)$. Then $f^+$ is a group homomorphism and the properties of morphisms in $Set^s$ yield the commutativity of the diagram (\ref{commutativity-morpphism-inversion maps}), so that $(R(X))f^+ \subseteq R(Y)$ holds; thus we get a well-defined group homomorphism $\tilde{f}$ induced by $f^+$.
\begin{definition}
Let $(X,S_X,i_X), (Y,S_Y,i_Y) $ be objects of $Set^s$ and $f: (X,S_X,i_X) \fto (Y,S_Y,i_Y)$ be a morphism; let also $\LL=(\LL, D, \Pi, j)$ be the unique smallest impartial subgroup of $G_X$ satisfying (1) and (2) of Lemma \ref{free-part-gr-str-set} with respect to $(X,S_X,i_X)$. Define the functor $F: Set^s \fto Part$ by
\[
F: (X,S_X,i_X) \mapsto (\LL, D, \Pi, j) \quad \text{and} \quad  fF:= \tilde{f} \vert_\LL,
\]
where $\tilde{f}: G_X \fto G_Y$ is the map defined above.
\end{definition}
It is a straightforward computation to verify that $F$ is a functor; as a consequence of the next lemma, we consider $F$ the free-construction functor. In order to simplify notation we simply write $X$ for an object $(X,S_X,i_X) \in Set^s$ whenever the context leaves no ambiguity.

\begin{lemma}\label{lemma-it is free}
	The following hold.
	\begin{itemize}
		\item[(a)] Let $(X,S_X,i_X)$ be an object in $Set^s$, then the partial group $(X)F =: (\LL, D, \Pi, j)$ defined in Lemma \ref{free-part-gr-str-set} is the free partial group over the object $(X,S_X,i_X)$.
		\item[(b)] The forgetful functor $U: Part \fto Set^s, \, (\M,D_\M,\Pi_\M,i_\M)U=(\M,D_\M,i_\M)$ is a right adjoint to the functor $F: Set^s \fto Part$.
	\end{itemize}
\end{lemma}

\begin{proof}
	\begin{itemize}
		\item[(a)] We follow the notation of Lemma \ref{free-part-gr-str-set}. Consider the inclusion map $\iota: (X,S_X,i_X) \hookrightarrow (X)FU = (\LL, D,j)$ naturally induced by the inclusion $X \subseteq \LL$; note here that $\iota^*$ over $S_X$ is the map $\nu$ defined after Lemma \ref{lemma:description of G_X}. Being the free partial group over the object $(X,S_X,i_X)$ means that for any other partial group $\M=(\M, D_\M, \Pi_\M, i_\M)$ and any morphism $f : (X,S_X,i_X) \fto (\M)U = (\M, D_\M,i_\M)$ in the category $Set^s$ there exists a unique morphism of partial groups $\tilde{f}: \LL \fto \M$ such that $f= \iota \cdot \tilde{f}U$, that is that makes the following diagram commutative.
		\begin{equation*}\label{diagram of lemma on free partial gr}
		\tag{$\Theta$}
			\begin{tikzcd}
			(X,S_X,i_X) \ar[rr,hook,"\iota"] \ar[rd,"\forall f"] &[-10pt] &[-10pt] (\LL,D, j) \ar[dl, dashed, "\tilde{f}U"] \\
				& (\M, D_\M, i_\M)
			\end{tikzcd}
		\end{equation*}
		
		Let $\mu: W(X) \fto G(X)$ be the monoid homomorphism defined in \ref{lemma:description of G_X}, consider the following diagram
		\begin{center}
				\begin{tikzcd}
				W(X) \ar[r,"f^*"] \ar[d,"\mu"] & W(\M) \supseteq D_\M \\
				G_X
				\end{tikzcd}
				\end{center}
		and define $\LL_0 := ((D_\M)(f^*)\inv)\mu \subseteq  G_X$. Since $\mu$ commutes with $i_X$ and $f$ is a morphism of $Set^s$, the inversion map on $G_X$ restricts to an involutory bijection on $\LL_0$. Denote $G_X=(G_X,W(G_X),\Pi_{G_X}, i_{G_X})$ and set
		\[
		D_0 := \{ w=(w_1,\dots,w_n) \in W(\LL_0) \mid \exists w'_i \in (w_i)\mu\inv \text{ s.t. } w'_1 \circ \dots \circ w'_n \in (D_\M)(f^*)\inv \};
		\]
		we show that $(\LL_0, D_0, \Pi_0, i_0)$ is a partial group, where $\Pi_0$ and $i_0$ are the restrictions to $D_0$, respectively $\LL_0$, of $\Pi_{G_X}$ and $i_{G_X}$. Note that, $\mu$ being a monoid homomorphism, for any $w \in D_0$ there exists $w' \in (D_\M)(f^*)\inv$ such that $(w)\Pi_0 = (w)\Pi_{G_X} = (w')\mu$; thus $(D_0)\Pi_0 \subseteq \LL_0$. Let now $u=(u_1,\dots,u_k), \, v=(v_1,\dots,v_r), \, w=(w_1,\dots,w_s) \in W(\LL_0)$; suppose that $u \circ v \in D_0$ and take preimages $u'_i, v'_j \in W(X)$ along $\mu$ with $u'_1 \circ \dots \circ u'_k \circ v'_1 \circ \dots \circ v'_r \in (D_\M)(f^*)\inv$, then $D_\M \ni (u'_1 \circ \dots \circ u'_k \circ v'_1 \circ \dots \circ v'_r)f^* = (u'_1 \circ \dots \circ u'_k)f^* \circ (v'_1 \circ \dots \circ v'_r)f^*$. Hence $(u'_1 \circ \dots \circ u'_k)f^* \in D_\M$ as well as $(v'_1 \circ \dots \circ v'_r)f^* \in D_\M$, so $u'_1 \circ \dots \circ u'_k, \, v'_1 \circ \dots \circ v'_r \in (D_\M)(f^*)\inv$, which implies $u,v \in D_0$. Axioms (1) and (2) of Definition \ref{def partial gr} now easily follow. In the same way one proves axioms (3) and (4).\\
		By restricting $f^*$ to the preimage $(D_\M)(f^*)\inv$ we may define the composition $f^* \Pi_\M$ as in the diagram below
		\begin{center}
						\begin{tikzcd}
						(D_\M)(f^*)\inv \ar[r,"f^*"] \ar[d,"\mu"] &  D_\M \ar[r,"\Pi_\M"] & \M \\
						\LL_0
						\end{tikzcd}
		\end{center}		
		Suppose there exist words $(x_1, \dots, x_n), \, (y_1,\dots, y_m) \in (D_\M)(f^*)\inv$ with $ (x_1, \dots, x_n)\mu = (y_1,\dots, y_m)\mu$, then by Lemma \ref{lemma:description of G_X}(b) $(x_1, \dots, x_n) \circ ( (y_1,\dots, y_m)i_X) \in R(X) \cap W(X)$. As a consequence $(x_1, \dots, x_n)f^*$ and $(y_1,\dots, y_m)f^*$ differ by an element of $(R(X))f^* \cap D_\M$. Since $(R(X) \cap (D_\M)(f^*)\inv)f^* \Pi_\M =\{ 1_\M \}$, we get $(x_1, \dots, x_n)f^*\Pi_\M =(y_1,\dots, y_m)f^*\Pi_\M$, so $f^* \Pi_\M$ (restricted to $(D_\M)(f^*)\inv$) factors via $\mu$ to a map $\tilde{f}_0 : \LL_0 \fto \M$.
		An immediate consequence of the definition of $D_0$ is that $(D_0)\tilde{f}_0^* \subseteq D_\M$, whereas  the commutativity of the diagram
		\begin{center}
		\begin{tikzcd}
		D_0 \ar[d,"\tilde{f}_0^*"]   \ar[r,"\Pi_0"] & \LL_0 \ar[d,"\tilde{f}_0"]   \\ 
		D_\M \ar[r,"\Pi_\M"]		 \ar[r,mapsto]  & \M
		\end{tikzcd}
		\end{center}
		follows from the definition of $\tilde{f}_0$, which is therefore a morphism of partial groups.\\
		Clearly $X \subseteq \LL_0$ and $S_X \subseteq D_0$; then Lemma \ref{free-part-gr-str-set} shows that also $\LL$ is embedded in $\LL_0$ and $D \subseteq D_0$, so we may restrict $\tilde{f}_0$ to a morphism
		\[
		\tilde{f} : \LL \fto \M,
		\]
		which makes (\ref{diagram of lemma on free partial gr}) commutative by construction. Uniqueness of $\tilde{f}$ follows since any other morphism of partial groups $g$ making (\ref{diagram of lemma on free partial gr}) commutative has to agree with $\tilde{f}$ over every element of $X$ in $\LL$ and therefore induces a map $g^*: W(X) \fto W(\M)$ equal to $f^*$.
		\item[(b)] The way of proceeding is exactly as in Lemma \ref{adjoint functors free-forgetful-first}, so we only sketch the proof. Let $X=(X,S_X,i_X), \, Y=(Y,S_Y,i_Y) \in Ob(Set^s)$ and $\LL=(\LL,D_\LL, \Pi_\LL,i_\LL), \M=(\M,D_\M,\Pi_\M,i_\M)$ be partial groups. The bijections 
		\[
		\psi_{X,\LL} : Hom_{Set^s}(X,(\LL)U) \fto Hom_{Part}((X)F,\LL)
		\]
		are a consequence of (a). Let $h:Y \fto X$ be a morphism in $Set^S$ and $g: \LL \fto \M$ a morphism in $Part$, then we must show that the following diagram is commutative
		\begin{center} 
				\begin{tikzcd}
				Hom_{Set^*}(X,(\LL)U) \ar[d] \ar[r,"\psi_{(X,\LL)}"] &[+20pt]	Hom_{Part}((X)F,\LL) \ar[d] \\
				Hom_{Set^*}(Y,(\M)U) \ar[r,"\psi_{(Y,\M)}"] &	Hom_{Part}((Y)F,\M)
				\end{tikzcd}
				\end{center}
		where the vertical morphisms are pre- and post-composition. Denoting $(f)\psi_{(X,\LL)} :=\tilde{f}$ (and similarly for the entire collection of maps), it means showing
		\[
		hF \cdot \tilde{f} \cdot g = \widetilde{h \cdot f \cdot gU}
		\]
		and the argument here relies again on looking at the action on elements of $Y$.
	\end{itemize}
\end{proof}

\subsection{Relations in partial groups}
If $G$ is a group and $x,y \in G$, it is always possible to rewrite a relation $x=y$ in the form $x y\inv = 1_G$; all relations are indeed identifications of elements with the unit $1_G$. Given a set $R$ of relations on $G$, one can identify the smallest normal subgroup $N$ of $G$ containing $R$ and form $G/N$ and one says that the relations in $R$ are satisfied by $G/N$. This line of reasoning is, however, not so straightforward for partial groups. Let $\LL =(\LL,D,\Pi,i)$ be a partial group, $x,y \in \LL$ and suppose that we would like to find a quotient of $\LL$ fulfilling the relation $x=y$: a priori there is no reason why the word $x \circ y\inv$ should be in $D$, however assuming that we can form such a quotient of $\LL$, axiom (4) of Definition \ref{def partial gr} yields $x \circ x\inv \in D$ and $(x \circ x\inv)\Pi=1$, so the relation $x = y$ becomes equivalent to $x \circ y\inv \in D$ and $(x \circ y\inv)\Pi = 1$. This shows that adding the element $x \circ y\inv$ to the domain of $\LL$ is a necessary condition for the relation $x=y$ to make sense.\\
Observe that relations on a group $G$ can equivalently be defined as equalities $w=u$ between words $w,u \in W(G)$. This suggests the following definition.
\begin{definition}
	Let $\LL=(\LL, D, \Pi, i)$ be a partial group, then a \emph{relation} on $\LL$ is an element of $D$.
\end{definition}

We now show that every partial group is a quotient of a free partial group.
\begin{proof}[Proof of Proposition \ref{prop:every-pg-is-quot-free}]
	Let $\LL=(\LL,D,\Pi,i)$ be a partial group and $U: Part \fto Set^s, \, (\LL,D,\Pi,i) \mapsto (\LL,D,i)$ the forgetful functor. Set $\M=(\M,D_\M, \Pi_\M, i_\M)$ to be the free partial group over $(\LL,D,i)$, according to Lemmas \ref{free-part-gr-str-set} and \ref{lemma-it is free}, and define the set of relations of $\LL$ as
	\[
	R_0:=\{u \in D \mid (u)\Pi =1_\LL \}.
	\]
	Since $D \subseteq D_\M$ by the definition of $\M$, the canonical inclusion $\LL \hookrightarrow \M, \, x \mapsto x$ extends to a map $\phi: D \mapsto \M, \, u \mapsto (u)\phi := (u)\Pi_\M$; then set $R:=(R_0)\phi$ and let $\N=(\N,D_\N, \Pi_\N,i_\N)$ be the partial subgroup of $\M$ generated by $R$ in the sense of \cite[Lemmas 1.8(e) and 1.9]{Chermak:2022}. We wish to prove that $\LL$ is isomorphic to the quotient $\M / \N$.\\
	The universal property of the free object $\M$ yields a map $f: \M \fto\LL$ extending the identity over $\LL$ (in particular, $f$ is surjective); let $i: \N \fto \M$ be the inclusion map and consider the pair $(\M/\N,q) = coker(i)$, as in the following diagram
	\begin{center} 
	\begin{tikzcd}
		&	\N \ar[d,hook,"i"] \\
	\LL \ar[r,hook] \ar[dr,"id_\LL"']	& \M \ar[d,two heads, "f"] \ar[r,"q"] & \M / \N  \ar[dl,two heads,"\hat{f}"] \\
	& \LL
	\end{tikzcd}
	\end{center}
	The map $f$ acts in the following way: suppose that $u \in \M$ and write it by representatives in $\LL$, so $u=u_1 \circ \dots \circ u_n$ with $u_i \in \LL$, then we may think at $u \in D_\M$ since $u=(u_1,\dots,u_n)\Pi_\M$, and $u =(u)f^* \in D$ and $(u)f =(u)\Pi$, since $f$ extends the identity map. In particular, every element of $\M$ can be identified with an element of $D$.	The morphism $\hat{f}$ is induced by the universal property of the cokernel $\M / \N$, since for every element $v \in R$, $(v)f=(v)\Pi = 1_\LL$ and $R$ generates $\N$. As $f$ is surjective, $\hat{f}$ is so as well.\\
	Clearly $(D(\M/\N))\hat{f}^* \subseteq D$ as $\hat{f}$ is a morphism in $Part$; if $u \in D$, viewing $u=(u)f^* \in D$ together with the commutativity of the diagram above (namely $f=q \hat{f}$) yields $(D(\M/\N))\hat{f}^* = D$. Hence we are only left with showing injectivity of $\hat{f}$; note that for any $u \in D$ axioms (3) and (4) of the definition of a partial group yield $(u)i \circ (u)\Pi \in D$ and $((u)i \circ (u)\Pi)\Pi = 1_\LL$, so that recalling $D \subseteq D_\M$ 
	\[
	u \circ (u)i\Pi \circ (u)\Pi \in D \quad \text{and} \quad (u \circ (u)i\Pi \circ (u)\Pi)\Pi_\M = u
	\]
	Moreover, $u \circ (u)i\Pi \in \N$. Suppose now that $u,v \in \M$ are such that $(u)f=(v)f$; we have already seen that this means $(u)\Pi = (v)\Pi$ (viewing $u,v$ as elements of $D$), so that
	\[
	u \circ (u)i\Pi \circ (u)\Pi \quad \text{and} \quad v \circ (v)i\Pi \circ (v)\Pi
	\]
	can be regarded as words of $D_\M$ of length two, namely $(u \circ (u)i\Pi, (u)\Pi)$ and $(v \circ (v)i\Pi, (v)\Pi)$. Let $\R$ be the equivalence relation defining $\M / \N$ as in Section \ref{cocompleteness}, then $(u \circ (u)i\Pi) \; \R \; (v \circ (v)i\Pi)$, since both are in $\N$, and $(u)\Pi =(v)\Pi$. As
	\[
	(u \circ (u)i\Pi \circ (u)\Pi)\Pi_\M = u \quad \text{and} \quad (v \circ (v)i\Pi \circ (v)\Pi)\Pi_\M =v
	\]
	by (\ref{condition of sim}) in Section \ref{cocompleteness} we get $u \; \R \; v$, that is $u = v$ in $\M / \N$. This completes the proof.
\end{proof}

In general, given a partial group $\LL=(\LL,D,\Pi,i)$, and a subset $S \subseteq W(\LL)$, one can freely add the relations in $S$ in the following way. Set $R:= \{w \in D \mid (w)\Pi =1_\LL\}$ and build the free partial group $\M^+=(\M^+,D^+,\Pi^+,i^+)$ over $(\LL,D \cup S, i)$; if $\N^+ \subseteq \M^+$ is the partial subgroup of $\M^+$ generated by elements of $R$ (seen as elements of $\M^+$ as in the previous lemma), then the domain of the quotient $\M^+ /\N^+$ contains the words in $S$. In particular, we have the following properties.

\begin{lemma}\label{lemma:adding relations}
Let $\LL$, $S$, $\M^+$, $R$ and $\N^+$ be as above, then $\M^+ /\N^+$ contains an isomorphic copy of $\LL$. If $\N'$ is the partial subgroup of $\M^+$ generated by the relations in $R \cup S$, then $\N'/\N^+$ is embedded in $\M^+/\N^+$ and $\M^+/\N'$ is isomorphic to $(\M^+/\N^+) / (\N'/\N^+)$.

\begin{proof}
Let $\M=(\M,D_\M,\Pi_\M, i_\M)$ be the free partial group over $(\LL,D,i)$ and $\N$ the partial subgroup of $\M$ generated by $R$. Then Proposition \ref{prop:every-pg-is-quot-free} shows that $\M / \N \cong \LL$; consider the diagram below
\begin{center}
\begin{tikzcd}
(\LL,D,i) \ar[r,hook] \ar[d,hook] & \M \ar[r,two heads,"q"] \ar[d,hook,"j"] & \M / \N \ar[d,"\iota"] \\
(\LL,D \cup S, i) \ar[r,hook] & \M^+ \ar[r, two heads,"q^+"] & \M^+ / \N^+ 
\end{tikzcd}
\end{center}
where $j$ is induced by the universal property of the free group $\M$ and $\iota$ by the universal property of the cokernel $\M/\N$, since $i_\N j q^+ = \hat{1}$ (as $(\N)j \subseteq \N^+$) for $i_\N$ the inclusion of $\N$ in $\M$. Since $\N$ and $\N^+$ are generated by the same elements, $\N$ is an impartial subgroup of $\M^+$ and $\N^+$ the partial subgroup generated by $\N$. Lemma \ref{lemma:quotients-part-impart-sbgr} shows that $\M^+/\N = \M^+/\N^+$, so $\iota$ is the inclusion morphism; since $\M/\N \cong \LL$, $\LL$ is embedded in $\M^+/\N^+$.\\
Consider now
\begin{center}
\begin{tikzcd}
\N' \ar[d,hook,"h"] \ar[r,two heads, "q^+|_{\N'}"] & \N' / \N^+ \ar[d,hook,"h'"] \\
\M^+ \ar[r, two heads,"q^+"] \ar[d,two heads, "g"] & \M^+ / \N^+  \ar[r,two heads,"t"] \ar[ld,two heads,"\eta"] & (\M^+/\N^+) \diagup (\N'/\N^+) \\
\M^+ / \N' 
\end{tikzcd}
\end{center}
where $t=coker(h')$ and $g=coker(h)$, whereas the map $\eta$ is induced by the universal property of the cokernel $\M^+/\N^+$. Since $h' \eta = \hat{1}$ and $h q^+ t = \hat{1}$, the universal properties of $(\M^+/\N^+) \diagup (\N'/\N^+)$ and $\M^+ / \N'$ yield maps $\psi, \phi: (\M^+/\N^+) \diagup (\N'/\N^+) \rightleftarrows \M^+ / \N'$ which are one inverse of the other, proving the wanted isomorphism.
\end{proof}
\end{lemma}

\subsection{Chermak's elementary expansions revisited}
As an example, we show that Chermak's construction in \cite[Section 5]{Chermak:2013}, renamed \emph{elementary expansion} in \cite[Section 3]{Chermak:2016:II}, is a special case of a construction made in terms of generators and relations; this was the fundamental step which allowed Chermak to build the centric linking system associated to a given saturated fusion system $\F$ and so eventually defining the classifying space of $\F$. Some notions about fusion systems together with a fairly broad background about localities are here required; we will in particular use results from \cite[Section 2, pages 57-64]{Chermak:2013} or, equivalently, from \cite[Sections 1 and 2]{Chermak:2022}. As a note to the reader, \cite[Section 3]{Chermak:2016:II} will be considered part of the background needed for the example; we will indeed adopt the same notation.\\

For the entire section $\LL=(\LL,\Delta,S)$ is a locality over the fusion system $\F$ in the sense of \cite[Remark 2.8 (2)]{Chermak:2013}, or equivalently as in \cite[Section 2]{Chermak:2016:II}, and write also $\LL=(\LL, D_\LL, \Pi, i)$. For every $g \in \LL$ Chermak defines $S_g := \{x \in S \mid (g\inv,x,g) \in D_\LL, \, x^g \in S \}$; in general, denote by $c_g$ the conjugation map by the element $g \in \LL$ and for every word $w=(w_1,\dots,w_n) \in W(\LL)$ write
\[
S_w := \{ x \in S \mid c_{w_1} \cdot \dots \cdot c_{w_n} \text{ is defined on $x$ and maps $x$ into $S$} \}.
\]

The goal consists in building a new locality $\LL^+=(\LL^+, D^+, \Pi^+, i^+)$ which has the same associated fusion system $\F$, but a larger set $\Delta^+$ of objects. Therefore suppose we wish to add the subgroup $R \le S$ to the set of objects, then the definition of objectivity, as in \cite[Definition 2.6]{Chermak:2013}, implies that we need to add also the following:
\begin{itemize}
\item[(i)] all the elements in $R^\F$ to $\Delta$;
\item[(ii)] new elements to $\LL$, which afford conjugations among the newly added objects.
\end{itemize}
This suggests to define $\Delta^+$ as the closure of $\Delta \cup R^\F$ in the set of subsets of subgroups of $S$, with respect to taking $\F$-conjugates and overgroups. We then build $\LL^+=(\LL^+, D^+, \Pi^+, i^+)$ as a quotient of the free partial group $\M = (\M, D_\M, \Pi_\M, i_\M)$ over the object $(\LL, T, i) \in Set^s$, with $T$ being a suitable subset $T \subset W(\LL)$ and $i$ the inversion map of $\LL$. Under few additional conditions (which were identified by Chermak) the partial group $\LL^+$ ends up possessing the structure of a locality of the form $(\LL^+,\Delta^+, S)$. The correct idea in order to identify $T$ is due to Chermak and consists in applying \cite[Corollary 2.7]{Chermak:2022}, so we choose
\[
T:=\{ w \in W(\LL) \mid S_w \in \Delta^+ \}.
\]

Indeed consider the partial group $(\M, D_\M, \Pi_\M,i_\M)$ just defined. If $S_w \in \Delta$, then $w \in D_\LL \subseteq D_\M$ (as $\Delta \subseteq \Delta^+$); on the other hand, suppose that $S_w \in \Delta^+ \setminus \Delta$ for $w = (w_1, \dots, w_k)$, then there exists an $\F$-conjugate $P_0$ of $R$ such that $P_0 \le S_{w_1}$ and so $w \in D_\M$ via $P_0$. In particular, all the elements $P_i:= P_{i-1}^{w_i}$ are $\F$ conjugates of $R$. Take now an $\F$ conjugate of $R$, say $(R)f$, and suppose that $f$ is the composition of restrictions of conjugacy maps induced by elements $(x_1, \dots,x_n) \in W(\LL)$ (recall that $\F$ is generated by such maps); then $\M$ has the property that $(x_1, \dots, x_n) \in D_\M$ with $(x_1, \dots, x_n)\Pi_\M  = x_1 \circ \dots \circ x_n$. In other words, all the $\F$-conjugations between elements of $R^\F$, and therefore also of $\Delta^+$, are realized as conjugations by some element of $\M$; this will translate in $(\LL^+,\Delta^+)$ being objective.\\

Just as in the proof of Proposition \ref{prop:every-pg-is-quot-free}, define the set $I$ of inner relations of $\LL$ by $I:= \{ u \in D \mid (u)\Pi = 1_\LL \}$ and set $\N := \langle I \rangle$, that is the partial subgroup of $\LL$ generated by $I$. Then we define
\[
\displaystyle \LL^+ := \frac{\M}{\N}
\]
By the definition of $T$ we have $D \subseteq T$, so Lemma \ref{lemma:adding relations} shows that $\LL$ is embedded in $\LL^+$. The following set of hypothesis was identified by Chermak in \cite[Hypothesis 3.2]{Chermak:2016:II} and guarantees that $(\LL^+, \Delta^+, S)$ is a locality.

\begin{hypothesis}\label{hyp:Chermak-for-locality}
The subgroup $R \le S$ satisfies the following:
\begin{itemize}
\item[(i)] every strict overgroup of $R$ contained in $S$ is an element of $\Delta$;
\item[(ii)] $R$ and $O_p(N_\F(R))$ are fully normalized in $\F$;
\item[(iii)] $N_\LL(R)$ is a subgroup of $\LL$ such that $N_\F(R) = \F_{N_S(R)}(N_\LL(R))$.
\end{itemize}
\end{hypothesis}

In particular, (i) implies $\Delta^+=\Delta \cup R^\F$; moreover the hypothesis of \cite[Lemma 3.1]{Chermak:2016:II} are fulfilled and points (b) and (c) of the lemma imply that for any $V \in R^\F$ there exists an element $y \in \LL$ such that $V = R^y$ and $N_S(V) \le S_{y\inv}$. Following the notation in \cite[Section 3]{Chermak:2016:II}, $\textbf{Y}_V$ denotes the set of such elements $y$ and $\textbf{X}_V := \{y\inv \mid y \in \textbf{Y}_V\}$.  Now for any element $g \in \LL$ there are two possible cases:
\begin{itemize}
\item[(a)] $S_g$ does not contain an element of $R^\F$; in this case the conjugation $c_g: S_g \fto S_g^g$ cannot be realized via an element of $R^\F$;
\item[(b)] $S_g$ contains an element $U \in R^\F$, then the word $g$ is in $D_\M$ via $(U,U^g)$.
\end{itemize}

The following lemma is the fundamental result connecting our construction with that of Chermak; one sees that the results which allowed Chermak to build a set and define a partial group structure are here used just to describe the partial group $\LL^+$. Most of the following proof consists indeed in the first lines of the proof of \cite[Lemma 3.8]{Chermak:2016:II}.
\begin{lemma} \label{lemma:description of elements of L plus}
The following hold.
\begin{itemize}
\item[(i)] If $g \in \LL$ is such that $S_g$ satisfies condition (b) above, then $g$ is equivalent to $(x\inv, h, y)$ in $\LL^+$ for suitable $x \in \textbf{X}_U$, $y \in \textbf{Y}_V$ and $h \in N_\LL(R)$, where $U,V \in R^\F$.
\item[(ii)] If $w=(w_1, \dots, w_n) \in \M$, then either each of the $w_i$ satisfies (b) (and we say that $w$ is \emph{of type (b)}), or none of them does; in the first case there exist suitable $x_i \in \textbf{X}_{U_i}$, $y_i \in \textbf{Y}_{V_i}$ and $h_i \in N_\LL(R)$ such that the word $w':=(x_1\inv,h_1,y_1,x_2\inv,h_2,y_2,\dots,x_n\inv,h_n,y_n) \in D_\M$ and it is equivalent to $w$ in $\LL^+$.
\item[(iii)] We have that each $(y_i,x_{i+1}\inv) \in D_\LL$ and also $w_0':=(h_1,y_1x_2\inv,h_2,y_2x_3\inv,\dots,y_{n-1}x_n\inv,h_n) \in D_\LL$; in particular, $\LL^+$ is finite.
\end{itemize}
\begin{proof}
Suppose that $g \in D_\M$ via $U \in R^\F$ and set $U^g =: V$; we have already seen that under the hypothesis \ref{hyp:Chermak-for-locality} the sets $\textbf{Y}_V$ and $\textbf{X}_U$ are not empty, thus pick $y \in \textbf{Y}_V$ and $x \in \textbf{X}_U$ and set $u:=(x,g,y\inv)$. Then
\[
u=(x,g,y\inv) \in D_\LL \qquad \text{via} \qquad (N_{S_g}(U)^{x\inv},N_{S_g}(U),N_{S_{g\inv}}(V),N_{S_{g\inv}}(V)^{y\inv}),
\]
so define $h:=(u)\Pi$ and note that $h \in N_\LL(R)$. Now the word $u':=(x\inv,x,g,y\inv,y) \in D_\LL$ via $N_{S_u}(U)$ and we get $g = (u')\Pi = (x\inv,h,y)\Pi$, thus proving (i).\\
If any of the $w_i$ for $i \in \{1,\dots,n\}$ is such that $S_{w_i}$ contains an $\F$-conjugate of $R$, than all the preceding and all the following $w_j$ share the same property, as the $S_{w_j}$ are, by definition, obtained by $\F$-conjugation. Suppose now that (b) holds for each $w_i$ and let $U_i$, $V_i$ be as $U$ and $V$ in (i), so that $w_i \in D_\M$ via $(U_i, V_i)$; since $w \in D_\M$ we can choose the sequence of the $(U_i,V_i)$ such that $U_i^{w_i} =V_i = U_{i+1}$. Since by (i) each $(x_i\inv, h_i, y_i)$ is equivalent to $w_i$ in $\LL^+$, $w'$ is equivalent to $w$; this is (ii).\\
Note now that each $(y_i, x_{i+1}\inv) \in D_\LL$ via $(N_S(V_i)^{y_i\inv},N_S(V_i)=N_S(U_{i+1}),N_S(U_{i+1})^{x_i\inv})$; in particular we get $(y_i,x_{i+1}\inv)\Pi_\M = (y_i, x_{i+1}\inv)\Pi = y_i x_{i+1}\inv $ in $\LL^+$ and, moreover, $y_i x_{i+1}\inv \in N_\LL(R)$. Then $w_0' \in W(N_\LL(R))$ and, $N_\LL(R)$ being a subgroup of $\LL$, it is a word of $D_\LL$. Moreover, $w = w' = (x_1\inv, (w_0')\Pi, y_n)$ in $\LL^+$. Let $\textbf{Y}:= \bigcup_{V \in R^\F} \textbf{Y}_V$ and similarly $\textbf{X} := \bigcup_{V \in R^\F} \textbf{X}_V$. If $w$ is such that no $w_i$ satisfies (b), then the word $(w_1, \dots, w_n)$ must be in $D_\M$ via some element of $\Delta$; in other words, $(w_1, \dots, w_n) \in D_\LL$, so $w =(w_1, \dots, w_n)\Pi$ in $\LL^+$. On the other hand, we have shown that whenever each $w_i$ satisfies (b), then $w$ is equivalent to an element appearing in the set $\textbf{X} \times N_\LL(R) \times \textbf{Y}$, which is finite being $\LL$ finite. This shows that also $\LL^+$ is finite, proving (iii).
\end{proof}
\end{lemma}
The lemma just proven shows actually more; it relates our construction with that of the set $\Phi$ and with the description given in \cite[Section 3 and description 3.9]{Chermak:2016:II} (or equivalently the set $\Theta$ in \cite[Section 5, Lemmas 5.8 and 5.9]{Chermak:2013}). \\
In \cite[Proposition 3.15(c)]{Chermak:2016:II} Chermak characterizes $\LL^+$ via a push-out diagram taken in $Part$. We show that our definition of $\LL^+$ coincides with that of Chermak by showing that also in our case $\LL^+$ is the push-out of the same diagram. Define $\LL_0 := \{ g \in \LL \mid S_g \text{ contains an $\F$-conjugate of $R$} \}$ and $D_0:=\{w \in D_\LL \mid S_w \text{ contains an $\F$-conjugate of $R$} \}$; it is not difficult to see that $\LL_0$, paired with the domain $D_0$, forms an impartial subgroup of $\LL$. Define $\LL_0^+ := \{ u \in \LL^+ \mid \text{$u$ is the image of an element of $\M$ of type (b)} \}$; then $\LL_0^+$ forms a partial subgroup of $\LL^+$. The embedding of $\LL$ in $\LL^+$ induces an embedding of $\LL_0$ in $\LL_0^+$, thus we get the following diagram 
\begin{equation*} \label{diagram:po} \tag{$\circ$}
\begin{tikzcd} 
\LL_0^+ \ar[r,hook,"\iota"] 	& \LL^+ \\
\LL_0 \ar[u,hook, "\lambda_0"] \ar[r,hook,"\iota_0"]	& \LL \ar[u,hook,"\lambda"]
\end{tikzcd}
\end{equation*}
where every morphism is an embedding. Our definition of $\LL_0$ clearly coincides with the same given in \cite[Lemma 3.14]{Chermak:2016:II}; as a consequence of Lemma \ref{lemma:description of elements of L plus} and of \cite[3.9]{Chermak:2016:II} also $\LL_0^+$ coincides with that defined by Chermak. Since the diagram (\ref{diagram:po}) is the same appearing in \cite[Proposition 3.15 (c)]{Chermak:2016:II}, we only need to show the following.
\begin{lemma}\label{lemma:equal construction of el.exp}
The diagram (\ref{diagram:po}) is a push-out diagram in $Part$.
\begin{proof}
Suppose to have partial group homomorphisms $f: \LL \fto \H$ and $g: \LL_0^+ \fto \H$, for some partial group $\H$, such that $\iota_0f = \lambda_0 g$; by Lemma \ref{lemma:description of elements of L plus}(ii) every element in $\LL^+$ is either an element of $\LL$ or an element of $\LL_0^+$, thus we can define
\[
\psi: \LL^+ \fto \H, \quad u \mapsto \begin{cases*}
										(u)g \quad & \text{if $u \in \LL_0^+$}\\
										(u)f \quad & \text{if $u \in \LL$}
										\end{cases*}
\]
Note that $\LL_0 = \LL_0^+ \cap \LL$ and that (\ref{diagram:po}) is clearly commutative, since all morphisms are inclusions; thus the map $\psi$ is well-defined. Every word in the domain $D_{\LL^+}$ of $\LL^+$ lifts to a representative in $D_\M$; the proof of Lemma \ref{lemma:description of elements of L plus} shows that either all representatives in $D_\M$ are of type (b), or none is. Then $\psi^*$ coincides with $f^*$ over words which lift to a representative not of type (b) and with $g^*$ over words lifting to representatives of type (b). This implies that $\psi$ is a morphism of partial groups.\\
Finally, uniqueness of the map $\psi$ is trivial; any other candidate map $\LL^+ \fto \H$ would restrict to $f$ over $\LL$ and to $g$ over $\LL_0^+$, coinciding then with $\psi$.
\end{proof}
\end{lemma}

The proof that $(\LL^+, D^+, S)$ is a locality is now given in \cite[Proposition 3.17]{Chermak:2016:II}, whereas \cite[Proposition 3.18]{Chermak:2016:II} shows uniqueness of such a locality.

\bibliographystyle{alpha}
\bibliography{rep-coh-local}

\end{document}